\documentclass[draft,reqno]{amsart}

\usepackage[Symbol]{upgreek}

\usepackage{amssymb}
\usepackage{epic}
\usepackage{mathrsfs}
\usepackage{accents}
\usepackage{stmaryrd}

\newcommand{\bbold}{\mathbb}

\def\R { {\bbold R} }
\def\Q { {\bbold Q} }
\def\Z { {\bbold Z} }

\def\N { {\bbold N} }
\def\T { {\bbold T} }

\def \ex{\operatorname{e}}

\renewcommand\epsilon{\varepsilon}

\def \<{\langle}
\def \>{\rangle}

\def \((  {(\!(}
\def \)) {)\!)}

\def \k {{{\boldsymbol{k}}}}

\DeclareMathSymbol{\precequ}{\mathrel}{symbols}{"16}
\DeclareMathSymbol{\succequ}{\mathrel}{symbols}{"17}

\newtheorem{theorem}{Theorem}[section]
\newtheorem{lemma}[theorem]{Lemma}
\newtheorem{prop}[theorem]{Proposition}
\newtheorem{cor}[theorem]{Corollary}

\theoremstyle{definition}

\theoremstyle{remark}

\newtheorem*{examples}{Examples}

\newtheorem*{notations}{Notations and conventions}

\newcommand{\abs}[1]{\lvert#1\rvert}

\def \Ga{\boldsymbol{\Gamma}}
\def \sc{\operatorname{sc}}

\let\oldi\i
\let\oldj\j
\renewcommand\i{\relax\ifmmode{\boldsymbol{i}}\else\oldi\fi}
\renewcommand\j{\relax\ifmmode{\boldsymbol{j}}\else\oldj\fi}

\renewcommand\leq{\leqslant}
\renewcommand\geq{\geqslant}
\renewcommand\preceq{\preccurlyeq}

\renewcommand\le{\leq}
\renewcommand\ge{\geq}

\DeclareMathAlphabet{\mathbf}{OML}{cmm}{b}{it}

\DeclareFontFamily{U}{fsy}{}
\DeclareFontShape{U}{fsy}{m}{n}{<->s*[.9]psyr}{}
\DeclareSymbolFont{der@m}{U}{fsy}{m}{n}
\DeclareMathSymbol{\der}{\mathord}{der@m}{182}

\DeclareSymbolFont{der@m}{U}{fsy}{m}{n}
\DeclareMathSymbol{\derdelta}{\mathord}{der@m}{100}

\DeclareSymbolFont{imag@m}{OT1}{cmr}{m}{ui}
\DeclareMathSymbol{\imag}{\mathord}{imag@m}{105}

\DeclareFontFamily{OMS}{smallo}{}
\DeclareFontShape{OMS}{smallo}{m}{n}{<->s*[.65]cmsy10}{}
\DeclareSymbolFont{smallo@m}{OMS}{smallo}{m}{n}
\DeclareMathSymbol{\smallo}{\mathord}{smallo@m}{79}

\DeclareFontFamily{OMS}{largerdot}{}
\DeclareFontShape{OMS}{largerdot}{m}{n}{<->s*[.8]cmsy10}{}
\DeclareSymbolFont{largerdot@m}{OMS}{largerdot}{m}{n}
\DeclareMathSymbol{\largerdot}{\mathord}{largerdot@m}{15}

\DeclareMathSymbol{\llambda}{\mathord}{der@m}{108}
\DeclareMathSymbol{\rrho}{\mathord}{der@m}{114}

\newcommand{\equationqed}[1]{\[\pushQED{\qed}#1 \qedhere\popQED\]\let\qed\relax}
\newcommand{\alignqed}[1]{\begin{align*}\pushQED{\qed} #1 \qedhere\popQED\end{align*}\let\qed\relax}

\makeatletter
\newcommand{\dminus}{\mathbin{\text{\@dminus}}}

\newcommand{\@dminus}{%
  \ooalign{\hidewidth\raise1ex\hbox{\bf.}\hidewidth\cr$\m@th-$\cr}%
}
\makeatother

\begin{document}

\title{Revisiting Closed Asymptotic Couples}

\author[Aschenbrenner]{Matthias Aschenbrenner}
\address{Kurt G\"odel Research Center for Mathematical Logic\\
Universit\"at Wien\\
1090 Wien\\ Austria}
\email{matthias.aschenbrenner@univie.ac.at}

\author[van den Dries]{Lou van den Dries}
\address{Department of Mathematics\\
University of Illinois at Urbana-Cham\-paign\\
Urbana, IL 61801\\
U.S.A.}
\email{vddries@illinois.edu}

\author[van der Hoeven]{Joris van der Hoeven}
\address{CNRS, LIX, \'Ecole Polytechnique\\
91128 Palaiseau Cedex\\
France}
\email{vdhoeven@lix.polytechnique.fr}

\begin{abstract} Every discrete definable subset of a closed asymptotic couple with ordered scalar field $\k$ is shown to be contained in a finite-dimensional $\k$-linear subspace of that couple.  It follows that the differential-valued field $\T$ of transseries induces more structure on its value group
than what is definable in its asymptotic couple equipped with its
scalar multiplication by real numbers, where this asymptotic couple is construed as a two-sorted structure 
with $\R$ as the underlying set for the second sort. 
\end{abstract}

\date{March 2022. {\it Competing interests}\/: The authors declare none}


\maketitle

\section*{Introduction}\label{intro}

\noindent
The field of Laurent series with real coefficients comes with a natural derivation but is too small to be closed under integration and exponentiation. These defects are cured by passing to a certain canonical extension, the ordered differential field~$\T$ of transseries. Transseries are formal
series in an indeterminate $x>\R$, such as
\begin{align*}
   &   - 3 \ex^{\ex^x} + \ex^{\textstyle\frac{\ex^x}{\log x} +
  \frac{\ex^x}{\log^2 x} + \frac{\ex^x}{\log^3 x} + \cdots} - x^{11} + 7\\
  &   \hspace{2em} + \frac{\pi}{x} + \frac{1}{x \log x} + \frac{1}{x \log^2
  x} + \frac{1}{x \log^3 x} + \cdots \nonumber\\
  &   \hspace{2em} + \frac{2}{x^2} + \frac{6}{x^3} + \frac{24}{x^4} +
  \frac{120}{x^5} + \frac{720}{x^6} + \cdots \nonumber\\
  &   \hspace{2em} + \ex^{- x} + 2 \ex^{- x^2} + 3 \ex^{- x^3} + 4 \ex^{-
  x^4} + \cdots, \nonumber
\end{align*}
where $\log^2 x := (\log x)^2$, etc.~Transseries, that is, 
elements of $\T$,  
are also the {\em lo\-ga\-rith\-mic-exponential series\/}  ($\operatorname{LE}$-series, for short)
from \cite{DMM}; we refer to that paper, or to Appendix~A of our book~\cite{ADH}, for a detailed construction of $\T$.

What we need for now is that $\T$ is a real closed field extension
of the field $\R$ of real numbers and that $\T$ comes equipped with 
a distinguished element $x>\R$, an exponential operation $\exp\colon \T \to \T$ and a distinguished derivation $\der\colon \T\to \T$. The exponentiation here is an isomorphism of the ordered additive group of $\T$ onto the ordered multiplicative group
$\T^{>}$ of positive elements of $\T$; we set $\ex^f:=\exp(f)$
for $f\in\T$. The derivation~$\der$
comes from
differentiating a transseries termwise with respect to $x$, and
we set $f':= \der(f)$, $f'':= \der^2(f)$, and so on, 
for $f\in \T$;
thus $x'=1$, and $\der$
is compatible with exponentiation: $(\ex^f)'=f'\ex^f$ for $f\in \T$. Moreover, the constant field of $\T$ is $\R$, that is,
$\{f\in \T:\,f'=0\}=\R$; see again \cite{ADH} for details. 
Before stating our new results we introduce some conventions: 

\begin{notations}
Throughout, $m$,~$n$ range over $\N=\{0,1,2,\dots\}$. Ordered sets, ordered abelian groups, and ordered fields are totally ordered, by convention. Given an ambient ordered set $S$, a {\em downward closed subset of $S$}, also called a
{\em cut in $S$}, is a set $D\subseteq S$ such that for all
$a,b\in S$ with $a < b\in D$ we have~$a\in D$.
For an (additively written) ordered abelian group $\Gamma$ we set $$\Gamma^{\ne}\ :=\ \Gamma\setminus \{0\}, \qquad
 \Gamma^{<}\ :=\ \{\gamma\in \Gamma:\,\gamma<0\}, \qquad
 \Gamma^{>}\ :=\ \{\gamma\in \Gamma:\,\gamma>0\}.$$
For any field $K$ we let $K^\times=K\setminus \{0\}$ be its multiplicative group.
A {\em differential field\/} is
a field $K$ of characteristic $0$ with a derivation 
$\der\colon K \to K$, and we set $a':= \der(a)$ for $a\in K$, and
let $b^\dagger:=b'/b$ be the logarithmic derivative
of $b\in K^\times$ when the ambient differential field $K$ with its derivation $\der$ is
clear from the context; note that then $(ab)^\dagger=a^\dagger + b^\dagger$ for $a,b\in K^\times$. 
\end{notations}

\noindent
Our book~\cite{ADH} culminated in an elimination theory for the differential field $\T$ of transseries. As a consequence
we found that the induced structure on its constant field $\R$
is just its semialgebraic structure: if $X\subseteq \R^n$ is definable in $\T$, then $X$ is semialgebraic 
(in the sense of $\R$). (Here and throughout ``definable in $\mathbf{M}$'' means ``definable in $\mathbf{M}$ with parameters from $\mathbf{M}$''.) 

The story is more complicated for the structure induced by $\T$ on its value group. To explain this, we recall that
the natural valuation ring 
$$\mathcal{O}_{\T}\ =\ 
\big\{f\in \T:\, \text{$|f|\le r$ for some real $r>0$}\big\}$$
of the real closed field $\T$ is clearly $0$-definable in $\T$ {\em as a differential field}, which is how we construe $\T$ in the rest of
this paper. Let
$v\colon \T^\times \to \Gamma_{\T}$ be the corresponding valuation
on the field $\T$. We may consider $\Gamma_{\T}$ as the quotient
${\T^\times}\!/\negmedspace \asymp$ and $v$ as the natural map to this quotient
where $\asymp$ is a $0$-definable equivalence relation
on $\T^\times$. 

Thus $\Gamma_{\T}$ is part of $\T^{\operatorname{eq}}$.
What is the structure induced by $\T$ on $\Gamma_{\T}$? It includes the structure of $\Gamma_{\T}$ as an ordered (by convention, additively written) abelian group.
Moreover, the derivation of $\T$ induces a function
$\psi\colon \Gamma_{\T}^{\ne} \to \Gamma_{\T}$ by
$\psi(vf)=v(f^\dagger)$ for $f\in \T^\times$ with $vf\ne 0$.
The structure $(\Gamma_{\T}, \psi)$ consisting of the ordered abelian group $\Gamma_{\T}$ with the function $\psi$ is the {\em asymptotic couple\/} of $\T$, a notion introduced for
differential-valued fields---among which is $\T$---by Rosenlicht~\cite{R}. There is also a natural
$0$-definable scalar multiplication 
$$(r,\gamma)\mapsto r\gamma:\,\R\times \Gamma_{\T}\to \Gamma_{\T}$$
that makes $\Gamma_{\T}$ into a vector space over $\R$; it is given by $rv(f)=v(f^r)$ for $f\in \T^{>}$, and the reason it is $0$-definable 
(in $\T^{\operatorname{eq}}$) is that $r\alpha=\beta$ (for $r\in \R$ and
$\alpha, \beta\in \Gamma_{\T}$) iff there are $f,g\in \T^{\times}$ such that $\alpha=vf$, $\beta=vg$ and $rf^\dagger=g^\dagger$. For this reason we consider the $2$-sorted structure 
$\Ga_{\T}=\big((\Gamma_{\T},\psi), \R; \sc\!\big)$
consisting of the asymptotic couple $(\Gamma_{\T},\psi)$, 
the field $\R$, and the above scalar multiplication
$$\sc\colon \R\times \Gamma_{\T}\to \Gamma_{\T}, \quad 
\sc(r,\gamma)=r\gamma.$$ 
The basic elementary properties of this $2$-sorted structure
were determined in~\cite{AvdD}. This structure encodes 
important features of $\T$, and in this paper we prove
 a new result about it in Section~\ref{chha}:

\begin{theorem}\label{thm} Let $\Gamma_{\T}$ be equipped with its order topology, and let $X\subseteq \Gamma_{\T}$ be definable in 
$\Ga_{\T}$. Then the following are equivalent:\begin{enumerate}
\item[(i)] $X$ is contained in a finite-dimensional $\R$-linear subspace of $\Gamma_{\T}$;
\item[(ii)] $X$ is discrete;
\item[(iii)] $X$ has empty interior in $\Gamma_{\T}$.
\end{enumerate}
\end{theorem}

\noindent
We also know from~\cite[Corollaries~14.3.10, 14.3.11]{ADH} that for any nonzero differential polynomial
$G(Y)\in \T\{Y\}$ the subset $\big\{vy:\,y\in \T^\times,\ G(y)=0\big\}$
of $\Gamma_{\T}$ is discrete. The set of zeros of $$G(Y):= Y^2Y'Y^{(3)}-Y^2(Y^{(2)})^2-Y(Y')^2Y^{(2)}+(Y')^4$$ in $\T$ is
$$\big\{a\ex^{b\ex^{cx}}:\,a,b,c\in \R\big\}\cup \big\{a\ex^{bx}:\,a,b\in \R\big\}.$$
For this $G$ the set $\big\{vy:\,y\in \T^\times,\ G(y)=0\big\}$ is not contained in a finite-dimensional $\R$-linear subspace of $\Gamma_{\T}$ and thus {\em not\/} 
definable in the $2$-sorted structure $\Ga_{\T}$
by the theorem above. We treat this example in more detail at the end of Section~\ref{prelim}.  

The authors of \cite{AvdD} had speculated that the subsets of 
$\Gamma_{\T}$ definable in $\T^{\operatorname{eq}}$
might be just those that are definable in the $2$-sorted
structure $\Ga_{\T}$. The above is a counter example,
but leaves open the possibility that
$\Gamma_{\T}$ is {\em stably embedded\/} in $\T^{\operatorname{eq}}$. In this connection we note that for all intents and purposes
we can replace the $2$-sorted structure
$\Ga_{\T}$ by the $1$-sorted structure $(\Gamma_{\T};\psi,\R 1,\sc)$
consisting of the asymptotic couple
$(\Gamma_{\T}; \psi)$ expanded by the set
$\R 1\subseteq \Gamma_{\T}$, where $1=v(x^{-1})\in \Gamma_{\T}^{>}$ is the unique fixed point of $\psi$, and by the function
$$\sc\ :\,(\R 1)\times \Gamma_{\T}\to \Gamma_{\T}, \quad \sc(r1,\gamma) := r\gamma.$$

\subsection*{Why revisit closed asymptotic couples?} The proof of Theorem~\ref{thm}  requires the results of \cite{AvdD}, suitably extended. This was our original motive for revisiting the subject of closed asymptotic couples. 
The theorem itself is of interest, but is also needed for its application to the
induced structure on the value group of $\T$. 

The quantifier elimination (QE) for closed asymptotic couples in \cite{AvdD} was expected to help in obtaining a
QE for $\T$. The latter is achieved in
\cite[Chapter 16]{ADH}, but there we needed only a key lemma from \cite{AvdD}, not its QE for closed asymptotic couples. 
That key lemma is \cite[Property B]{AvdD}, and is given a self-contained proof of five dense pages in \cite[Section~9.9]{ADH}. Since then we found a 
simpler way to obtain the QE in~\cite{AvdD} that does not use the key lemma alluded to, but depends on some easier-to-prove
new lemmas that have also other applications; see Section~\ref{ext}. This new proof of QE, given in Section~\ref{qe},
is another reason for revisiting the subject of closed asymptotic couples.
(We derive the ``key lemma'' itself as a routine consequence of the QE for closed asymptotic couples: Proposition~\ref{cac} below.) 

For his study of transexponential pre-$H$-fields in~\cite[Chapter~6]{NPC1} and~\cite{NPC2}, Nigel Pynn-Coates introduced a modified version of ``closed asymptotic couple''  and adapted accordingly some material from our 
(unpublished) 2017 version of this paper.  Getting the paper published is also more urgent now because in our recent proof that maximal Hardy fields are $\eta_1$  we use results from Section~\ref{sec:simple} below.

Finally, this paper gives us an opportunity to enhance and better organize parts of~\cite{AvdD}, and acknowledge gaps in some proofs there; we intend to close these gaps in a follow-up to the present paper. No familiarity with \cite{AvdD} is needed, but we do assume as background some 20 pages (mainly on asymptotic couples) 
from \cite{ADH}, namely parts of Section~2.4 on ordered abelian groups, Sections~6.5,~9.1~(subsection on asymptotic couples), 9.2 (first four pages), and 9.8. For the reader's convenience we also repeat definitions of key notions concerning asymptotic couples and $H$-fields.

\medskip\noindent
We thank Nigel Pynn-Coates for his careful reading of this paper, and corrections, and the referee for helpful comments.

\section{Preliminaries}\label{prelim}

\noindent
We only consider asymptotic couples of $H$-type, calling them
{\em $H$-couples\/} for brevity. Thus an $H$-couple is a pair
$(\Gamma,\psi)$ consisting of an ordered abelian group $\Gamma$
with a map $\psi\colon \Gamma^{\ne}\to \Gamma$, such that for all $\alpha,\beta\in \Gamma^{\ne}$, \begin{enumerate}
\item[(AC1)] $\alpha+\beta\ne 0\ \Longrightarrow\ \psi(\alpha+\beta)\ge \min\!\big(\psi(\alpha),\psi(\beta)\big)$;
\item[(AC2)] $\psi(k\alpha)=\psi(\alpha)$ for all $k\in \Z^{\ne}$;
\item[(AC3)] $\alpha>0\ \Longrightarrow 
\alpha+\psi(\alpha)> \psi(\beta)$;
\item [(HC)] $0<\alpha\le \beta\ \Longrightarrow\ \psi(\alpha)\ge \psi(\beta)$.
\end{enumerate}
(As an aside, note that (AC2) and (HC) together imply (AC1); had we observed this earlier, it would have shortened some arguments in \cite[Section~9.8]{ADH}; the reader can use it to the same effect in Section~\ref{ext} of the present paper.) 
Let $(\Gamma,\psi)$ be an $H$-couple. By (AC1) and (AC2) the function $\psi$ is a valuation on the
abelian group~$\Gamma$; as usual we extend $\psi$ to 
$\psi\colon \Gamma\to \Gamma_{\infty}:=\Gamma\cup \{\infty\}$ by $\psi(0):= \infty$; we use $\alpha^\dagger$ as an alternative notation for $\psi(\alpha)$ and set $\alpha':=\alpha+\alpha^\dagger$ 
for $\alpha\in \Gamma$. Also $\Psi:= \psi(\Gamma^{\ne})$. We recall from \cite[Corollary 9.2.16]{ADH} a basic trichotomy for $H$-couples which says that we are in
exactly one of the following three cases: \begin{itemize}
\item $(\Gamma,\psi)$ has a (necessarily unique) {\em gap}, that is, an element $\gamma\in \Gamma$
such that~$\Psi < \gamma< (\Gamma^{>})'$;
\item $(\Gamma,\psi)$ is {\em grounded}, that is, $\Psi$ has a largest element;
\item $(\Gamma,\psi)$ has {\em asymptotic integration}, that is,
$\Gamma=(\Gamma^{\ne})'$.
\end{itemize}
We say that $(\Gamma,\psi)$ is {\em closed\/} if $\Gamma$ is divisible, $\Psi\subseteq \Gamma$ is downward closed, and $(\Gamma,\psi)$ has asymptotic integration. We also use the qualifiers {\em having a gap}, {\em grounded}, {\em having asymptotic integration}, and {\em closed\/} for $H$-couples with extra structure. 

An {\bf $H$-cut\/} in $(\Gamma,\psi)$ is a downward closed set 
$P\subseteq \Gamma$ such that $\Psi\subseteq P<(\Gamma^{>})'$. 
The set $\Psi^{\downarrow}:=\{\alpha\in \Gamma:\, \text{$\alpha\le\beta$ for some $\beta\in \Psi$}\}$ is an $H$-cut in $(\Gamma,\psi)$, and
if $(\Gamma,\psi)$ is grounded or has asymptotic integration, this
is the only $H$-cut in $(\Gamma,\psi)$. If $(\Gamma,\psi)$
has a gap $\beta$, then $\Psi^{\downarrow}\cup \{\beta\}$ is the only other $H$-cut in $(\Gamma,\psi)$. 
 
In particular, if $(\Gamma,\psi)$ is closed, then $\Psi$ is the only $H$-cut in $(\Gamma,\psi)$, but in eliminating quantifiers for closed $H$-couples in Section~\ref{qe} it is essential to have
a predicate for this $H$-cut in our language.

\subsection*{Where do closed $H$-couples come from?}
We recall from \cite[Chapter~10]{ADH} that an {\it $H$-field}\/ is an ordered differential field $K$ with constant field~$C$ such that: \begin{enumerate}
\item[(H1)] $a'>0$ for all $a\in K$ with $a>C$;
\item[(H2)] $\mathcal{O}=C+\smallo$, where $\mathcal{O}$ is the
convex hull of $C$ in the ordered field $K$, and $\smallo$ is the maximal ideal of the valuation ring $\mathcal{O}$.
\end{enumerate}
Let $K$ be an $H$-field, and let $\mathcal{O}$ and $\smallo$ be as in (H2). Thus $K$ is a valued field with valuation ring $\mathcal{O}$. Let $v\colon K^\times \to \Gamma$ be the associated valuation. The value group~$\Gamma=v(K^\times)$ is made into an
$H$-couple $(\Gamma,\psi)$---the $H$-couple of $K$---by $\psi(vf):=v(f^\dagger)$ for $f\in K^\times$ with $vf\ne 0$. We call $K$ {\it Liouville closed}\/ if it is real closed and for all~$a\in K$ there exists $b\in K$ with $a=b'$ and also a $b\in K^\times$ such that $a=b^\dagger$. 

{\it If $K$ is Liouville closed, its $H$-couple is closed\/} as is easily verified. We recall from \cite{ADH} that 
$\T$ is a Liouville closed $H$-field.

\subsection*{Ordered vector spaces}
Throughout we let $\k$, $\k_0$, and $\k^*$ be ordered fields.
Recall that an {\it ordered vector space over $\k$\/} is an ordered 
abelian group
$\Gamma$ with a scalar multiplication $\k\times \Gamma\to \Gamma$
that makes $\Gamma$ into a vector space over $\k$ such that~$c\gamma>0$ for all $c\in \k^{>}$ and $\gamma\in \Gamma^{>}$.
Let $\Gamma$ be an ordered vector space over $\k$. Then any $\k$-linear subspace of $\Gamma$ is considered as an ordered vector space over $\k$ in the obvious way. We shall need the following easy result about $\Gamma$: 

\begin{lemma}\label{clovs} Let $\Gamma_0$ be a $\k$-linear subspace of $\Gamma$. Suppose $\Gamma$ contains an element $\varepsilon$ with $0 < \varepsilon < \Gamma_0^{>}$. Then $\Gamma_0$ is closed in $\Gamma$ with respect to the order topology on $\Gamma$.
\end{lemma}
\begin{proof} Let $\gamma\in \Gamma\setminus \Gamma_0$.
With $\varepsilon$ as in the hypothesis we observe that the in\-ter\-val~${(\gamma-\varepsilon, \gamma+\varepsilon)}$ can have at most one point in it from $\Gamma_0$, and so by decreasing
$\varepsilon$ we can arrange that $(\gamma-\varepsilon, \gamma+\varepsilon)\cap \Gamma_0=\emptyset$.
\end{proof}

\noindent
The {\bf $\k$-archi\-mede\-an class\/} of $\alpha\in \Gamma$ is
$$[\alpha]_{\k}\ :=\ \big\{\gamma\in \Gamma:\,
\text{$\abs{\gamma} \leq c\abs{\alpha}$ and 
$\abs{\alpha} \leq c\abs{\gamma}$ for some $c\in \k^{>}$}\big\}.$$ 
Let $[\Gamma]_{\k}$ be the set of $\k$-archimedean classes. Then $[\Gamma]_{\k}$ is a partition
of $\Gamma$, and we linearly order $[\Gamma]_{\k}$ by
\begin{align*}
[\alpha]_{\k}<[\beta]_{\k} &\quad:\Longleftrightarrow\quad 
\text{$c\abs{\alpha}<\abs{\beta}$ for all $c\in \k^{>}$} \\
&\quad\hskip0.3em \Longleftrightarrow\quad \text{$[\alpha]_{\k}\neq [\beta]_{\k}$ and $\abs{\alpha}<\abs{\beta}$.}
\end{align*} 
Thus $[0]_{\k}=\{0\}$ is the smallest $\k$-archimedean class.
For $\alpha,\beta\in\Gamma$, $c\in\k^\times$ we have~$[c\alpha]_{\k}=[\alpha]_{\k}$ and
$[\alpha+\beta]_{\k} \leq \max\big([\alpha]_{\k},[\beta]_{\k}\big)$, with equality if
$[\alpha]_{\k}\neq[\beta]_{\k}$.

\begin{lemma}\label{ovsdis1} Let $\Gamma\ne \{0\}$ be an ordered vector space over $\k$ such that $[\Gamma^{\ne}]_{\k}$ has no least element. Then every finite-dimensional $\k$-linear subspace of $\Gamma$ is discrete with respect to the order topology on $\Gamma$.
\end{lemma}
\begin{proof} First note that if $\gamma_1,\dots,\gamma_n\in \Gamma^{\ne}$ and $[\gamma_1]_{\k},\dots,[\gamma_n]_{\k}$ are distinct, then $\gamma_1,\dots,\gamma_n$ are $\k$-linearly independent. 
Thus for a finite-dimensional $\k$-linear subspace $\Delta\ne \{0\}$ of $\Gamma$ we can take $\delta\in \Delta^{\ne}$ such that $[\delta]_{\k}$ is minimal in $[\Delta^{\ne}]_{\k}$. Then
for any $\alpha\in \Delta$ and $\beta\in \Gamma^{\ne}$ with
$[\beta]_{\k}<[\delta]_{\k}$ we have $\alpha+\beta\notin \Delta$. 
\end{proof}

\noindent
Lemma~\ref{ovsdis1} takes care of the easy direction (i)~$\Rightarrow$~(ii) in Theorem~\ref{thm}. The direction~(ii)~$\Rightarrow$~(iii)
is trivial. The harder direction (iii)~$\Rightarrow$~(i) uses a generality on expanded vector spaces, to which we now turn. 

Let
$V$ be a vector space over a field $C$. We consider
the two-sorted structure~$(V,C;\sc)$ consisting of the abelian
group $V$, the field $C$, and the scalar multiplication 
$\sc\colon C\times V \to V$ of the vector space $V$. Let $X\subseteq V$.
Then we have the expansion $\mathbf{V}=\big((V,X),C;\sc\!\big)$ of 
$(V,C;\sc)$. Let $\mathbf{V}^*=\big((V^*, X^*), C^*;\sc\!\big)$ be an elementary extension of  $\mathbf{V}$. Let $C^* V$ be the
$C^*$-linear subspace of $V^*$ spanned by~$V$.

\begin{lemma}\label{ovsdis2} Assume $\mathbf{V}^*$ is $|V|^+$-saturated. Then $X$ is contained in a finite-dimensional $C$-linear subspace of $V$ if and only if $X^*\subseteq C^*V$.
\end{lemma} 
\begin{proof} If $X\subseteq Cv_1+\cdots + Cv_n$, $v_1,\dots, v_n\in V$, then $X^*\subseteq C^*v_1+\cdots + C^*v_n\subseteq C^*V$.
We prove the contrapositive of the other direction, so assume
$X\not\subseteq Cv_1+\cdots +Cv_n$ for all $v_1,\dots, v_n\in V$. Then $X^*\not\subseteq C^*v_1+\cdots +C^*v_n$ for all $v_1,\dots, v_n\in V$, and so by saturation we get an element of $X^*$ that does not lie in $C^*V$.
\end{proof}

\noindent
For certain $(V,C;\sc)$ this will be applied to sets $X\subseteq V$ that are definable
in a suitable expansion of $(V,C;\sc)$, with $X^*$ the corresponding set in an elementary extension of that expansion.

\subsection*{$H$-couples over ordered fields}
Ordered vector spaces come into play as follows. Let $K$ be a Liouville closed $H$-field. It has the (ordered) constant field 
$C$, and the $H$-couple $(\Gamma,\psi)$. We have a map 
$(c,\gamma)\mapsto c\gamma\colon C\times \Gamma\to \Gamma$ such that
$cvf=vg$ whenever $f,g\in K^\times$ and 
$cf^\dagger=g^\dagger$. This map makes $\Gamma$ into an ordered vector space over $C$, and $\psi(c\gamma)=\psi(\gamma)$
for all $c\in C^\times$ and $\gamma\in \Gamma^{\ne}$. 

Accordingly, we define an $H$-couple over $\k$ to be an $H$-couple
$(\Gamma,\psi)$ where the ordered abelian group $\Gamma$ is also
equipped with a map $\k\times \Gamma\to \Gamma$ making
$\Gamma$ into an ordered vector space over $\k$ such that
$\psi(c\gamma)=\psi(\gamma)$ for all $c\in \k^\times$ and $\gamma\in \Gamma^{\ne}$. Thus the $H$-couple of a Liouville closed
$H$-field is naturally an $H$-couple over its constant field.

\medskip\noindent
Let $(\Gamma,\psi)$ be an $H$-couple over $\k$. A basic fact is that for distinct
$\alpha,\beta\in \Gamma^{\ne}$ we have 
$\big[\psi(\alpha)-\psi(\beta)\big]_{\k} < [\alpha-\beta]_{\k}$,
since for all $c\in \k^{>}$ we have $\psi(\alpha)-\psi(\beta)=\psi(c\alpha)-\psi(c\beta)=o\big(c(\alpha-\beta)\big)$, by \cite[6.5.4(ii)]{ADH}. 
Note also that for all $\alpha,\beta\in \Gamma^{\ne}$, 
$$[\alpha]_{\k}=[\beta]_{\k}\ \Longrightarrow\ \psi(\alpha)=\psi(\beta).$$

\subsection*{Hahn spaces} These are the ordered Hahn spaces from
\cite[Section~2.4]{ADH}: a Hahn space $\Gamma$ over $\k$ is an ordered vector space over $\k$ such that for all~$\alpha,\beta\in \Gamma^{\neq}$ with~$[\alpha]_{\k}=[\beta]_{\k}$ there exists 
$c\in \k^\times$ such that $[\alpha-c\beta]_{\k}<[\alpha]_{\k}$. 

\begin{examples}
\mbox{}

\begin{enumerate}
\item Any $1$-dimensional ordered vector space over $\k$ is a Hahn space over $\k$.
\item Any $\k$-linear subspace of a Hahn space over $\k$ is a Hahn space over $\k$.
\item Any ordered vector space over the ordered field $\R$ is a Hahn space over
$\R$.
\item The ordered $\Q$-vector space $\Q+\Q\sqrt 2 \subseteq \R$ is not a
Hahn space over $\Q$.

\end{enumerate}
\end{examples}

\noindent
We say that an $H$-couple $(\Gamma,\psi)$ over $\k$ is of {\bf Hahn type\/} if for all  $\alpha,\beta\in \Gamma^{\ne}$ with~$\psi(\alpha)=\psi(\beta)$ there exists a scalar $c\in \k$ such that $\psi(\alpha-c\beta)> \psi(\alpha)$; equivalently, 
$\Gamma$ is a Hahn space over $\k$ and for all $\alpha,\beta\in \Gamma^{\ne}$, 
$$\psi(\alpha)=\psi(\beta)\ \Longrightarrow\ 
[\alpha]_\k=[\beta]_\k.$$ 
Let $K$ be a Liouville closed $H$-field. We made its $H$-couple $(\Gamma, \psi)$ into an $H$-couple over its constant field $C$, and
as such $(\Gamma,\psi)$ is of Hahn type.

\subsection*{Details on the example in the introduction} We consider the Liouville closed $H$-field $\T$ and its element 
$x$ with $x'=1$. For $z\in \T$ with $z'\notin \R$ we have
\begin{align*} zz''\ =\ (z')^2\ &\Longleftrightarrow\ z^\dagger\ =\ (z')^\dagger\ \Longleftrightarrow\ (z'/z)^\dagger\ =\ 0\ \Longleftrightarrow\ 
z'=tz \text{ for some }t\in \R^\times\\
&\Longleftrightarrow\ z=s\ex^{tx} \text{ for some }s,t\in \R^\times.
\end{align*}
Considering also the case where $z'\in \R$ we conclude that
$$\big\{z\in \T:\, zz''=(z')^2\big\}\ =\ \big\{s\ex^{tx}:\, s,t\in \R\big\}.$$
Next, let $y\in \T^\times$ and suppose $z:=y^\dagger$ satisfies
$zz''=(z')^2$. Then $y=r\ex^{u}$ for some~$r\in \R$ and $u\in \T$ with $u'=z$. For $z=s\ex^{tx}$ with $s,t\in \R$ and
$u\in \T$, $u'=z$ we get $u\in \R\ex^{tx}+\R$ if $t\ne 0$, and
$u\in \R x+\R$ if $t=0$. Hence $y=a\ex^{b\ex^{cx}}$ or $y=a\ex^{bx}$ for some $a,b,c\in \R$. From $zz''=(z')^2$ we get
$$y^2y'y^{(3)}-y^2(y^{(2)})^2-y(y')^2y^{(2)}+(y')^4=0.$$
In this way we get for $$G(Y):=Y^2Y'Y^{(3)}-Y^2(Y^{(2)})^2-Y(Y')^2Y^{(2)}+(Y')^4$$ that its set of zeros in $\T$ is
$$\big\{a\ex^{b\ex^{cx}}:\, a,b,c\in \R\big\}\cup 
\big\{a\ex^{bx}:\, a,b\in \R\big\}.$$ 
It is easy to see that
for $0<c < d$ in $\R$ we have $\big[v(\ex^{\ex^{cx}})\big]_{\R} < \big[v(\ex^{\ex^{dx}})\big]_{\R}$, so the set~$\big\{vy:\, y\in \T^\times,\ G(y)=0\big\}$ is not contained in a
finite-dimensional $\R$-linear subspace of $\Gamma_{\T}$.

\section{Extensions of $H$-couples}\label{ext} 

\noindent {\em In this section 
$(\Gamma,\psi)$  and $(\Gamma_1, \psi_1)$ are 
$H$-couples over $\k$}.  An
{\bf embedding} $$h \colon\ (\Gamma,\psi)
\to (\Gamma_1,\psi_1)$$ is an embedding $h \colon\Gamma\to\Gamma_1$
of ordered vector spaces over $\k$
such that $$ h\bigl(\psi(\gamma)\bigr)\ =\ \psi_1\bigl(h(\gamma)\bigr)\ 
\text{ for $\gamma\in\Gamma^{\ne}$.}$$
If $\Gamma\subseteq \Gamma_1$ and the inclusion $\Gamma\hookrightarrow\Gamma_1$ is 
an embedding $(\Gamma,\psi)
\to (\Gamma_1,\psi_1)$, then we call~$(\Gamma_1,\psi_1)$ an {\bf extension} of $(\Gamma,\psi)$. If $(\Gamma_1, \psi_1)$ is of Hahn type and extends 
$(\Gamma, \psi)$, then $(\Gamma, \psi)$ is of Hahn type.

\subsection*{Embedding lemmas} The lemmas in this subsection are the analogues for $H$-couples over $\k$ of similar lemmas for
$H$-couples 
in \cite[Section~9.8]{ADH}. The proofs are essentially the same, 
so we omit them.

\begin{lemma}\label{extension1}
Let $\beta$ be a gap in $(\Gamma, \psi)$. Then there is an 
$H$-couple ${(\Gamma+\k \alpha,
\psi^{\alpha})}$ over $\k$ that extends $(\Gamma,\psi)$ such that: \begin{enumerate}
\item[\textup{(i)}] $\alpha>0$ and $\alpha'=\beta$;
\item[\textup{(ii)}] if $i\colon (\Gamma,\psi)\to (\Gamma_1,\psi_1)$ is an embedding and $\alpha_1\in\Gamma_1$, $\alpha_1>0$, $\alpha_1'=i(\beta)$, then
$i$ extends uniquely to an embedding $j\colon\bigl(\Gamma+\k \alpha,
\psi^{\alpha}\bigr)\to (\Gamma_1,\psi_1)$ with~$j(\alpha)=\alpha_1$.
\end{enumerate}
\end{lemma}

\noindent
The universal property (ii) determines $(\Gamma+\k\alpha, \psi^{\alpha})$ 
up to isomorphism over $(\Gamma, \psi)$, and
$0<c\alpha < \Gamma^{>}$ for all $c\in \k^{>}$; moreover,
for all $\gamma\in \Gamma$ and $c\in \k$ with $\gamma+c\alpha\ne 0$,
\begin{equation}\label{eq:psialpha}
\psi^{\alpha}(\gamma+c \alpha)\ =\ \begin{cases}
\psi(\gamma), &\text{if $\gamma\neq 0$,}\\
\beta-\alpha, &\text{otherwise.}
\end{cases} 
\end{equation}
Note also that 
$[\Gamma+ \k\alpha]_{\k}=[\Gamma]_{\k} \cup \big\{[\alpha]_{\k}\big\}$, so for $\Psi^{\alpha}:=\psi^{\alpha}\big((\Gamma+\k\alpha)^{\ne}\big)$ we have:
\begin{equation}\label{eq:Psialpha} \Psi^{\alpha}\ =\ \Psi\cup\{\beta-\alpha\}, \qquad  \max\Psi^{\alpha}\ =\ \psi^{\alpha}(\alpha)\ =\ \beta-\alpha.
\end{equation}
Lemma~\ref{extension1} goes through with $\alpha < 0$ and $\alpha_1 <0$ in place of $\alpha >0$ and $\alpha_1 > 0$, respectively. In the setting of this modified lemma
we have $\Gamma^{<} < c\alpha < 0$
for all~$c\in \k^{>}$, \eqref{eq:psialpha} goes through for $\gamma\in \Gamma$ and $c\in \k$ with $\gamma + c\alpha \ne 0$, \eqref{eq:Psialpha} goes through. So we have two ways to
remove a gap. 
Removal of a gap as above leads by \eqref{eq:Psialpha} to a grounded $H$-couple over $\k$, and this is the situation we consider next. 

\begin{lemma}\label{extas2} Assume that $\Psi$ has a largest element $\beta$. 
Then there exists an 
$H$-couple~${(\Gamma+\k \alpha,
\psi^{\alpha})}$ over $\k$ that extends $(\Gamma,\psi)$ with $\alpha\ne 0$, $\alpha'=\beta$, such that for any embedding~$i\colon (\Gamma,\psi)\to (\Gamma_1,\psi_1)$ and any $\alpha_1\in\Gamma_1^{\ne}$ with $\alpha_1'=i(\beta)$ there is a unique extension of~$i$ to an embedding $j\colon(\Gamma+\k \alpha,
\psi^{\alpha})\to (\Gamma_1,\psi_1)$ with $j(\alpha)=\alpha_1$.
\end{lemma}

\medskip\noindent
Let $(\Gamma+\k\alpha, \psi^{\alpha})$ be as in Lemma~\ref{extas2}. Then $\Gamma^{<} < c\alpha < 0$
for all~$c\in \k^{>}$,
$[\Gamma+\k\alpha]_{\k}=[\Gamma]_{\k} \cup \big\{[\alpha]_{\k}\big\}$, so \eqref{eq:Psialpha}
holds for $\Psi^{\alpha}:=\psi^{\alpha}\big((\Gamma+\k\alpha)^{\ne}\big)$. 
Thus our new $\Psi$-set $\Psi^{\alpha}$ still has a maximum, but
this maximum is larger than the maximum $\beta$ of the original 
$\Psi$-set $\Psi$. By iterating this construction indefinitely
and taking a union, we obtain an
$H$-couple over $\k$ with asymptotic integration.

Once we have an $H$-couple over $\k$ with asymptotic integration, we can create an extension with a gap as follows:

\begin{lemma}\label{addgap} Suppose that $(\Gamma, \psi)$ has asymptotic 
integration. 
Then there is an $H$-couple 
$(\Gamma+ \k\beta, \psi_{\beta})$ over $\k$ extending $(\Gamma, \psi)$ such that: \begin{enumerate}
\item[\textup{(i)}] $\Psi< \beta < (\Gamma^{>})'$;
\item[\textup{(ii)}] for any $(\Gamma_1, \psi_1)$ extending 
$(\Gamma, \psi)$ and $\beta_1\in \Gamma_1$ with  $\Psi< \beta_1 < (\Gamma^{>})'$
there is a unique embedding 
$(\Gamma+ \k\beta, \psi_{\beta})\to (\Gamma_1, \psi_1)$ of 
$H$-couples over $\k$ that is the 
identity on $\Gamma$ and sends $\beta$ to $\beta_1$. 
\end{enumerate}
\end{lemma}

\noindent
Let $(\Gamma, \psi)$ and $(\Gamma+ \k\beta, \psi_{\beta})$ be as in Lemma~\ref{addgap}. If 
$(\Gamma + \k\alpha, \psi_{\alpha})$ is also an $H$-couple over 
$\k$
extending $(\Gamma, \psi)$ with $\Psi< \alpha < (\Gamma^{>})'$, then
by (ii) we have an isomorphism~$(\Gamma+ \k\beta, \psi_{\beta})\to (\Gamma+ \k\alpha, \psi_{\alpha})$ 
of $H$-couples over $\k$
that is the identity on $\Gamma$ and sends $\beta$ to $\alpha$.
In this sense, $(\Gamma + \k\beta, \psi_{\beta})$ is unique up to isomorphism
over~$(\Gamma, \psi)$. The construction of $(\Gamma+ \k\beta, \psi_{\beta})$ gives the following extra information,
with $\Psi_{\beta}$ the set of values of $\psi_{\beta}$ on 
$(\Gamma+ \k\beta)^{\ne}$: 

\begin{cor}\label{addgapcor} The set $\Gamma$ is dense 
in the ordered abelian group 
$\Gamma+ \k\beta$, so 
$[\Gamma]_{\k} = [\Gamma+ \k\beta]_{\k}$, 
$\Psi_{\beta}=\Psi$ and $\beta$ is a gap in 
$(\Gamma+ \k\beta, \psi_{\beta})$.  
\end{cor}

\noindent 
Recall that a cut in an 
ordered set $S$ is just a downward closed subset of $S$, and that 
an element $a$ of an ordered set extending $S$ is said to
realize a cut $D$ in $S$ if~$D<a<S\setminus D$ (so $a\notin S$).

\begin{lemma}\label{extension5}
Let $D$ be a cut in $[\Gamma^{\ne}]_{\k}$ and let $\beta\in\Gamma$ be such that
$\beta<(\Gamma^{>})'$,
$\gamma^\dagger \leq\beta$ for all $\gamma\in\Gamma^{\ne}$ with $[\gamma]_{\k}> D$,
and $\beta\leq \delta^\dagger$ for all $\delta\in\Gamma^{\ne}$ with $[\delta]_{\k}\in D$. 
Then there 
exists an $H$-couple $(\Gamma\oplus \k\alpha,\psi^\alpha)$
over $\k$ that 
extends
$(\Gamma,\psi)$, with $\alpha>0$, such that:
\begin{enumerate}
\item[\textup{(i)}] $[\alpha]_\k$ realizes the cut $D$ in $[\Gamma^{\ne}]_\k$, and 
$\psi^\alpha(\alpha)=\beta$;
\item[\textup{(ii)}] for any embedding $i\colon (\Gamma,\psi)\to
(\Gamma_1,\psi_1)$ 
and $\alpha_1\in \Gamma_1^{>}$ such that $[\alpha_1]_\k$ realizes 
the cut $\bigl\{\bigl[i(\delta)\bigr]_\k:
[\delta]_\k\in D\bigr\}$
in $\bigl[i(\Gamma^{\ne})\bigr]_\k$ and $\psi_1(\alpha_1)=i(\beta)$, $i$ extends uniquely to an embedding $j\colon 
{(\Gamma\oplus \k\alpha,\psi^\alpha)}\to (\Gamma_1,\psi_1)$ with $j(\alpha)=
\alpha_1$.
\end{enumerate}
Moreover, $[\Gamma\oplus\k\alpha]_{\k}=[\Gamma]_{\k} \cup 
\big\{[\alpha]_{\k}\big\}$ and $\Psi^{\alpha}:=\psi^{\alpha}\big((\Gamma\oplus \k\alpha)^{\ne}\big)=\Psi\cup\{\beta\}$.  If
$(\Gamma, \psi)$ is grounded, then so is $(\Gamma\oplus \k\alpha, \psi^\alpha)$. If $(\Gamma,\psi)$ has asymptotic integration, then so does
$(\Gamma\oplus \k\alpha, \psi^\alpha)$. If $\beta\in \Psi^{\downarrow}$, then a gap in $(\Gamma,\psi)$ remains a gap
in $(\Gamma\oplus \k\alpha, \psi^{\alpha})$. 
\end{lemma}
\begin{proof} By a straightforward analogue of \cite[Lemma~2.4.5]{ADH} we   extend $\Gamma$ to an ordered vector space
$\Gamma^{\alpha}=\Gamma\oplus \k\alpha$ over $\k$ with $\alpha>0$ such that 
$[\alpha]_{\k}$ realizes the cut~$D$ in~$[\Gamma^{\ne}]_{\k}$. Then $[\Gamma\oplus\k\alpha]_{\k}=[\Gamma]_{\k} \cup 
\big\{[\alpha]_{\k}\big\}$. We extend $\psi$ to $\psi^{\alpha}\colon (\Gamma^{\alpha})^{\ne} \to \Gamma$ by 
$$\psi^{\alpha}(\gamma+c\alpha)\ :=\ \min\!\big\{\psi(\gamma),\beta\big\}\ \text{ for $\gamma\in \Gamma$, $c\in \k^\times$.} $$
Apart from some obvious changes we now follow the proof of \cite[Lemma~9.8.7]{ADH}. This gives the desired results, except for the last Claim~of the lemma. To prove that claim, 
let $\beta\in \Psi^{\downarrow}$, let $\gamma\in \Gamma$ be a gap in $(\Gamma,\psi)$, and assume towards a contradiction that
$\gamma$ is not a gap in $(\Gamma^{\alpha},\psi^{\alpha})$. Then 
$\gamma> \Psi^\alpha$, so $\gamma=(\delta+c\alpha)'$ with 
$\delta\in \Gamma$, $c\in \k^\times$ and
$0 <\delta+c\alpha < \Gamma^{>}$. Then $[\delta+c\alpha]_{\k}\notin [\Gamma]_{\k}$, so $[\delta+c\alpha]_{\k}=[\alpha]_{\k}$.
As $\Psi$ has no largest element, we get $\Psi < (\delta+c\alpha)^\dagger=\alpha^\dagger=\beta$,
a contradiction.   
\end{proof}

\subsection*{The case of Hahn type} In Lem\-ma~\ref{extension1} (and in its variant with $\alpha<0$), in Lem\-ma~\ref{extas2}, and in Lem\-ma~\ref{extension5} for $\beta\notin \Psi$, we have: 
$$ \text{\em if $(\Gamma, \psi)$ is of Hahn type, then so is $(\Gamma+\k\alpha, \psi^{\alpha})$}.$$ 
Suppose $(\Gamma, \psi)$ and $(\Gamma+\k\beta, \psi_{\beta})$ are as in Lemma~\ref{addgap}, 
and $(\Gamma,\psi)$ is of Hahn type. 
We claim that then $(\Gamma+\k\beta, \psi_{\beta})$ is also of Hahn type. To prove this claim, recall from Corollary~\ref{addgapcor}
that $\Gamma$ is dense in $\Gamma+\k\beta$.  It follows easily that for
 nonzero~$\alpha_1, \alpha_2\in \Gamma+\k\beta$ with $\psi_{\beta}(\alpha_1)=\psi_{\beta}(\alpha_2)$ we have $[\alpha_1]_{\k}=[\alpha_2]_{\k}$. It remains to show that $\Gamma+\k\beta$ is a Hahn space over $\k$. So let  $\alpha_1, \alpha_2\in  \Gamma+\k\beta$ be nonzero with $[\alpha_1]_{\k}=[\alpha_2]_{\k}$. By density again, and the fact that $[\Gamma]_{\k}=[\Gamma+\k\beta]_{\k}$ has no least element $> [0]_{\k}$, 
we have $\gamma_1, \gamma_2\in \Gamma$ such that $[\alpha_1-\gamma_1]_{\k}< [\alpha_1]_{\k}$ and $[\alpha_2-\gamma_2]_{\k}< [\alpha_2]_{\k}$. Take $c\in \k^\times$ such that $[\gamma_1-c\gamma_2]_{\k}< [\gamma_1]_{\k}$.
It follows easily that then $[\alpha_1-c\alpha_2]_{\k}< [\alpha_1]_{\k}$.

\subsection*{New extension lemmas} The three next lemmas will enable in the next section a simpler proof of QE for closed
$H$-couples than in \cite{AvdD}: in that paper we needed ``properties (A) and (B)'' with long and tedious proofs,
and here we avoid this. 

\begin{lemma}\label{g1g2} Suppose $(\Gamma_1,\psi_1)$ extends 
$(\Gamma,\psi)$. Let $\beta\in \Gamma_1\setminus \Gamma$ and 
$\alpha_0\in \Gamma$ be such that 
$(\beta-\alpha_0)^\dagger\notin \Gamma$. Then 
$(\beta-\alpha_0)^\dagger\ =\max\big\{(\beta-\alpha)^\dagger:\,\alpha\in \Gamma\big\}$.  If in addition $\Gamma^{<}$ is cofinal in $\Gamma_1^{<}$, then
$(\beta-\alpha_0)^\dagger\le \text{ some element of }\Psi$. 
\end{lemma}
\begin{proof} Suppose $\alpha\in \Gamma$ and
$(\beta-\alpha)^\dagger > (\beta-\alpha_0)^\dagger$.
Then $\alpha-\alpha_0=(\beta-\alpha_0)-(\beta-\alpha)$ gives 
$(\beta-\alpha_0)^\dagger=(\alpha-\alpha_0)^\dagger\in \Gamma$, a contradiction. Assume $|\beta-\alpha_0|\ge |\gamma|$, $\gamma\in \Gamma^{\ne}$. Then $(\beta-\alpha_0)^\dagger\le \gamma^\dagger\in \Psi$. 
\end{proof}

\begin{lemma}\label{cutinheritance} Suppose $(\Gamma,\psi)$ is closed and $(\Gamma_1, \psi_1)$ and $(\Gamma_*,\psi_*)$ are $H$-couples over~$\k$ extending $(\Gamma,\psi)$.  Let $\beta\in \Gamma_1\setminus \Gamma$ and $\beta_*\in \Gamma_*\setminus \Gamma$ realize the same cut in $\Gamma$, and suppose that $\beta^\dagger\notin \Gamma$ and $\Gamma^{<}$ is cofinal in
$(\Gamma+\k\beta^\dagger)^{<}$.
 Then $\beta_*^\dagger\notin \Gamma$, and
$\beta^\dagger$ and $\beta_*^\dagger$
realize the same cut in $\Gamma$. 
\end{lemma} 
\begin{proof} 
Let $\alpha\in \Gamma^{\ne}$. We claim:
$$\beta^\dagger<\alpha^\dagger\ \Rightarrow\ \beta_*^\dagger<\alpha^\dagger, \qquad \beta^\dagger>\alpha^\dagger\ \Rightarrow\ \beta_*^\dagger>\alpha^\dagger.$$
To prove the first implication, assume $\beta^\dagger<\alpha^\dagger$. Then $|\beta|>|\alpha|$, so 
$|\beta_*|>|\alpha|$, and thus 
$\beta_*^\dagger\le\alpha^\dagger$. Since $(\Gamma,\psi)$ is closed and $\Gamma^{<}$ is cofinal in $(\Gamma+\k\beta^\dagger)^{<}$, we can replace in this
argument $\alpha$ by some $\gamma\in \Gamma^{\ne}$ with
$\beta^\dagger<\gamma^\dagger<\alpha^\dagger$, to get
$\beta_*^\dagger\le \gamma^\dagger<\alpha^\dagger$, and thus 
$\beta_*^\dagger<\alpha^\dagger$ as claimed.
The second implication follows in the same way. 

If
$\beta^\dagger< \gamma^\dagger$ for some 
$\gamma\in \Gamma^{\ne}$, then
$(\Gamma,\psi)$ being closed gives the desired conclusion. If $\beta^\dagger >\Psi$, then we use instead
$\Psi < \beta^\dagger < (\Gamma^{>})'$ and $\Psi < \beta_*^\dagger < (\Gamma^{>})'$.
\end{proof} 


\begin{lemma}\label{b1b2} Suppose $(\Gamma_1,\psi_1)$ extends $(\Gamma,\psi)$. Let $\beta\in \Gamma_1\setminus \Gamma$ and $\alpha_0, \alpha_1\in \Gamma$ be such that $\beta_0^\dagger\notin \Gamma$ for $\beta_0:= \beta-\alpha_0$ and $\beta_1^\dagger\notin \Psi$
for $\beta_1:=\beta_0^\dagger-\alpha_1$. Assume also that~$|\beta_0|\ge |\alpha|$ for some $\alpha\in \Gamma^{\ne}$.
Then $\beta_0^\dagger < \beta_1^\dagger$. 
\end{lemma}
\begin{proof} From $|\beta_0|\ge |\alpha|$ with $\alpha\in \Gamma^{\ne}$  we get $\beta_0^\dagger\le \alpha^\dagger$.
Also, 
$\beta_0^\dagger-\alpha^\dagger\notin \Gamma$ and~${[\beta_0^\dagger-\alpha_1]_{\k}}\notin [\Gamma]_{\k}$, hence 
$[\beta_0^\dagger-\alpha^\dagger]_{\k}\ge 
[\beta_0^\dagger-\alpha_1]_{\k}$. In view of \cite[6.5.4(i)]{ADH} this gives
\equationqed{\beta_0^\dagger\ =\ \min(\beta_0^\dagger, \alpha^\dagger)\ <\ (\beta_0^\dagger-\alpha^\dagger)^\dagger\ \le\ (\beta_0^\dagger-\alpha_1)^\dagger\ =\ \beta_1^\dagger.}
\end{proof}

\section{Eliminating Quantifiers for Closed $H$-couples}\label{qe}

\noindent
Eliminating quantifiers for closed $H$-couples requires a predicate for their $\Psi$-set, and in this connection we need to study the substructures of the thus expanded
$H$-couples. Accordingly we define
an {\em $H$-triple over $\k$\/} to be a triple 
$(\Gamma, \psi, P)$ where~$(\Gamma,\psi)$ is an $H$-couple 
over $\k$ and $P\subseteq \Gamma$ is an $H$-cut in $(\Gamma,\psi)$. 

\begin{lemma}\label{extpsi} Let $(\Gamma,\psi, P)$ be an $H$-triple over $\k$, and let $\beta\in P\setminus \Psi$. Then $(\Gamma,\psi, P)$ can be extended to an $H$-triple $(\Gamma\oplus \k\alpha, \psi^{\alpha}, P^{\alpha})$ over $\k$ such that: \begin{enumerate}
\item[\textup{(i)}] $\alpha>0$ and $\psi^{\alpha}(\alpha)=\beta$;
\item[\textup{(ii)}] given any embedding $i\colon (\Gamma, \psi, P)\to (\Gamma^*, \psi^*, P^*)$ and any element $\alpha^*>0$ in $\Gamma^*$ with
$\psi^*(\alpha^*)=i(\beta)$, there is a unique extension of $i$ to an embedding~$j\colon (\Gamma\oplus \k\alpha, \psi^{\alpha}, P^{\alpha})\to (\Gamma^*, \psi^*, P^*)$ with $j(\alpha)=\alpha^*$.
\end{enumerate}
If $(\Gamma,\psi)$ is of Hahn type, then so is $(\Gamma\oplus \k\alpha, \psi^{\alpha})$. 
\end{lemma}
\begin{proof} Distinguishing various cases this follows from Lemma~\ref{extension5}, especially the claims beginning with ``Moreover''. Use also ``The case of Hahn type''. 
\end{proof}

\noindent
An {\em $H$-closure}\/ of an $H$-triple $(\Gamma, \psi, P)$ over $\k$ is defined to be a closed $H$-triple $(\Gamma^{\operatorname{c}}, \psi^{\operatorname{c}}, P^{\operatorname{c}})$ over $\k$ that extends $(\Gamma,\psi,P)$ such that any embedding 
$$(\Gamma, \psi, P)\ \to\ (\Gamma^*, \psi^*, P^*)$$ into a closed $H$-triple $(\Gamma^*, \psi^*, P^*)$ over $\k$ extends to an embedding 
$$(\Gamma^{\operatorname{c}}, \psi^{\operatorname{c}}, P^{\operatorname{c}})\ \to\ (\Gamma^*, \psi^*, P^*).$$

\begin{cor}\label{hcl} Every $H$-triple over $\k$ has an $H$-closure. Every $H$-triple over $\k$ of Hahn type has an $H$-closure that is of Hahn type. 
\end{cor}
\begin{proof} This is a straightforward consequence of Lemmas~\ref{extension1}, \ref{extas2}, and \ref{extpsi}, using for the second statement also the
remarks in ``The case of Hahn type''.  
\end{proof}

\noindent
We consider $H$-triples as $\mathcal{L}_{\k}$-structures where $\mathcal{L}_{\k}$ is
the natural language of ordered vector spaces
over $\k$, augmented by a constant symbol $\infty$, a unary function symbol~$\psi$, and a unary relation symbol $P$. The underlying set of an $H$-triple 
$(\Gamma, \psi, P)$, when construed as an $\mathcal{L}_{\k}$-structure, is $\Gamma_{\infty}$ rather than $\Gamma$, and
the symbols of $\mathcal{L}_{\k}$ are interpreted in $(\Gamma, \psi, P)$ as usual, with $\infty$ serving as a default value: $$\psi(0)\  =\ \psi(\infty)\ =\ \gamma+\infty\ =\ \infty+\gamma\ 
=\ \infty + \infty\ =\ -\infty\ =\ c\infty\ =\ \infty$$ for $\gamma\in \Gamma$ and $c\in \k$. Also $0^\dagger:=\infty$ for the zero element $0\in \Gamma$, so $\Gamma^\dagger=\Psi\cup\{\infty\}$.

\begin{theorem}\label{hclQE} The $\mathcal{L}_{\k}$-theory of closed $H$-triples over $\k$ has $\operatorname{QE}$.
\end{theorem}

\subsection*{The proof of QE} Towards Theorem~\ref{hclQE}
we consider an $H$-triple $(\Gamma, \psi, P)$ over~$\k$ and
closed $H$-triples $(\Gamma_1, \psi_1, P_1)$ and $(\Gamma_*, \psi_*, P_*)$ over $\k$ that extend $(\Gamma, \psi, P)$, and such that $(\Gamma_*, \psi_*, P_*)$ is $|\Gamma|^+$-saturated. For $\gamma\in \Gamma_1$ we
let $\big(\Gamma\<\gamma\>, \psi_{\gamma}\big)$ be the $H$-couple over $\k$ generated by
$\Gamma\cup \{\gamma\}$ in $(\Gamma_1, \psi_1)$, and set 
$P_{\gamma}:= P_1\cap \Gamma\<\gamma\>$. 

Let $\beta\in \Gamma_1\setminus \Gamma$.
Theorem~\ref{hclQE} 
follows if  we can show that under these assumptions~$\big(\Gamma\<\beta\>,\psi_{\beta}, P_{\beta}\big)$ can be embedded over $\Gamma$ into $(\Gamma_*, \psi_*, P_*)$. We first do this in a situation that may seem rather special:

\begin{lemma}\label{noextradaggers} Suppose $(\Gamma, \psi)$ has asymptotic integration and $(\Gamma+ \k\beta)^\dagger=\Gamma^\dagger$.
Then~$\big(\Gamma\<\beta\>,\psi_{\beta}, P_{\beta}\big)$ can be embedded over $\Gamma$ into $(\Gamma_*, \psi_*, P_*)$.
\end{lemma}
\begin{proof} From 
$(\Gamma+ \k\beta)^\dagger=\Gamma^\dagger$ we get
$\Gamma\<\beta\>=\Gamma+\k\beta$. We have six cases:

\medskip\noindent
{\em Case~1}: $(\Gamma^{>})^\dagger < \eta < (\Gamma^{>})'$
and $\eta\in P_1$ for some $\eta\in \Gamma+\k\beta$. Fix such $\eta$. Then~$\Gamma$ is dense in $\Gamma+\k\eta=\Gamma+\k\beta$, by Corollary~\ref{addgapcor}, so
$[\Gamma+\k\beta]_{\k}=[\Gamma]_{\k}$. Moreover, there is no $\eta_1\ne \eta$ in $\Gamma+\k\beta$ with $(\Gamma^{>})^\dagger< \eta_1 < (\Gamma^{>})'$. 
By saturation we can take~$\eta_*\in \Gamma_*$ such that
$(\Gamma^{>})^\dagger < \eta_* < (\Gamma^{>})'$ and
$\eta_*\in P_*$. Then 
\cite[2.4.16]{ADH} yields
an embedding~$i \colon \Gamma+ \k\beta\to \Gamma_*$
of ordered vector spaces over $\k$ that is the identity on~$\Gamma$
with $i(\eta)=\eta_*$. This $i$ embeds 
$\big(\Gamma\<\beta\>,\psi_{\beta}, P_{\beta}\big)$ into
$(\Gamma_*, \psi_*, P_*)$.

\medskip\noindent
{\em Case~2}: $(\Gamma^{>})^\dagger < \eta < (\Gamma^{>})'$
and $\eta\notin P_1$ for some $\eta\in \Gamma+\k\beta$. Fixing such $\eta$, we repeat the argument of Case~1, except that now $\eta_*\notin P_*$ instead of $\eta_*\in P_*$.  
 
\medskip\noindent
{\em Case~3}: $[\Gamma+\k\beta]_{\k}=[\Gamma]_{\k}$, but there is no 
$\eta\in \Gamma+\k\beta$ with 
$(\Gamma^{>})^\dagger < \eta < (\Gamma^{>})'$. Then
$P_{\beta}$ is the only $H$-cut of $\Gamma\<\beta\>$.
Saturation yields $\beta_*\in \Gamma_*$ realizing the same cut in
$\Gamma$ as $\beta$. Then \cite[2.4.16]{ADH} yields
an embedding $i \colon \Gamma+ \k\beta\to \Gamma_*$
of ordered vector spaces over $\k$ that is the identity on 
$\Gamma$
with $i(\beta)=\beta_*$. For 
$\gamma\in \Gamma+\k\beta$ we have
$\big[i(\gamma)\big]_{\k}=[\gamma]_{\k}\in [\Gamma]_{\k}$, so 
$i(\gamma)^\dagger=\gamma^\dagger\in \Gamma^\dagger$. Thus $i$ embeds $\big(\Gamma\<\beta\>,\psi_{\beta}, P_{\beta}\big)$ into~$(\Gamma_*, \psi_*, P_*)$. 

\medskip\noindent
Assume next that we are not in Case~1, or Case~2, or Case~3. Then $[\Gamma+ \k\beta]_{\k}\ne[\Gamma]_{\k}$. Take $\gamma\in \Gamma\<\beta\>\setminus \Gamma$ such that $\gamma>0$ and 
$[\gamma]_{\k}\notin [\Gamma]_{\k}$, so $\big[\Gamma\<\beta\>\big]_{\k}=[\Gamma]_{\k} \cup \big\{[\gamma]_{\k}\big\}$. We are not in 
Case~1 or Case~2, so $P_{\beta}$ is the only $H$-cut of $\big(\Gamma\<\beta\>, \psi_{\beta}\big)$.  Let $D$ be the cut in $\Gamma$ realized by $\gamma$ and $E:= \Gamma\setminus D$, so
$D<\gamma < E$. Then $D$ has no largest element, and so $D\cap \Gamma^{>}\ne \emptyset$: if $d=\max D$, then we have $0 <\gamma-d< \Gamma^{>}$, and thus~$(\Gamma^{>})^\dagger < (\gamma-d)^\dagger < (\Gamma^{>})'$, contradicting that we are not in Case~1. Likewise,~$E$~has no least element. Here are the remaining cases:

\medskip\noindent
{\em Case~4}: 
$\gamma^\dagger\in (D^{>0})^\dagger\cap E^\dagger$. 
Saturation yields $\gamma_*\in \Gamma_*$ realizing the same cut in $\Gamma$ as~$\gamma$. Then $\gamma_*^\dagger=\gamma^\dagger\in (D^{>0})^\dagger$, and \cite[2.4.16]{ADH} yields
an embedding $i \colon \Gamma+ \k\beta\to \Gamma_*$
of ordered vector spaces over $\k$ that is the identity on 
$\Gamma$
with $i(\gamma)=\gamma_*$; this
 $i$ embeds $\big(\Gamma\<\beta\>,\psi_{\beta}, P_{\beta}\big)$ into $(\Gamma_*, \psi_*, P_*)$. 
 
\medskip\noindent
{\em Case~5}: 
$\gamma^\dagger\in (D^{>0})^\dagger > E^\dagger$. 
Then saturation yields a $\gamma_*\in \Gamma_{*}$
realizing the same cut in $\Gamma$ as $\gamma$, with $\gamma_*^\dagger=\gamma^\dagger$. By \cite[2.4.16]{ADH} this yields
an embedding $i \colon \Gamma+ \k\beta\to \Gamma_*$
of ordered vector spaces over $\k$ that is the identity on 
$\Gamma$
with $i(\gamma)=\gamma_*$, and so as before
 $i$ embeds $\big(\Gamma\<\beta\>, \psi_{\beta}, P_{\beta}\big)$ into
$(\Gamma_*, \psi_*, P_*)$.

\medskip\noindent
{\em Case~6}: $\gamma^\dagger\in E^\dagger < (D^{>0})^\dagger$.
This is handled just like Case~5. 
\end{proof}

\noindent
Note that Cases~4,~5,~6 in the proof above do not occur if 
$(\Gamma_1,\psi_1)$ is of Hahn type.

\medskip\noindent
In view of
Corollary~\ref{hcl} and Lemma~\ref{noextradaggers}, Theorem~\ref{hclQE} reduces to: 

\begin{lemma}\label{cinf} Suppose $(\Gamma,\psi)$ is closed and 
$(\Gamma+ \k\gamma)^\dagger\ne \Gamma^\dagger$ for all $\gamma\in \Gamma_1\setminus \Gamma$. Then~$\big(\Gamma\<\beta\>, \psi_{\beta}, P_{\beta}\big)$ embeds into $(\Gamma_*, \psi_*, P_*)$ over 
$\Gamma$.
\end{lemma} 
\begin{proof} If $\gamma\in \Gamma_1\setminus \Gamma$ and 
$\Psi< \gamma < (\Gamma^{>})'$, then $(\Gamma+ \k\gamma)^\dagger= \Gamma^\dagger$, contradicting our assumption. Hence there is no
such $\gamma$. It follows that $\Gamma^{<}$ is cofinal in 
$\Gamma_1^{<}$.

Take $\alpha_0\in \Gamma$ such that $(\beta-\alpha_0)^\dagger\notin \Gamma^\dagger$. Since $(\Gamma,\psi)$ is closed, this means ${(\beta-\alpha_0)^\dagger\notin \Gamma}$. Next take 
$\alpha_1\in \Gamma$ with $\big((\beta-\alpha_0)^\dagger-\alpha_1)^\dagger\notin \Gamma^\dagger$. Continuing this way we obtain sequences $\alpha_0, \alpha_1, \alpha_2,\dots$ in $\Gamma$ 
and $\beta_0, \beta_1, \beta_2,\dots$ in $\Gamma\<\beta\>\setminus \Gamma$
with
$$\beta_0\ =\ \beta-\alpha_0, \qquad \beta_{n+1}\ =\ \beta_{n}^\dagger-\alpha_{n+1} \text{ for all }n,$$
such that $\beta_n^\dagger\notin \Gamma$ for all $n$. By Lemma~\ref{b1b2} we have $\beta_0^\dagger < \beta_1^\dagger < \beta_2^\dagger < \cdots$. It follows that $[\beta_0]_{\k} >[\beta_1]_{\k} >[\beta_2]_{\k} > \cdots$, with $[\beta_n]_{\k}\notin [\Gamma]_{\k}$ for all $n$. In particular,
the family $(\beta_n)$ is $\k$-linearly
independent over $\Gamma$, and 
$$\Gamma\<\beta\>\ =\ \Gamma\oplus \k\beta_0 \oplus \k\beta_1 \oplus \k\beta_2\oplus \cdots.$$
By saturation we can take $\beta_*\in \Gamma_*\setminus \Gamma$
realizing the same cut in $\Gamma$ as $\beta$. This gives an embedding $e_0\colon \Gamma\oplus \k \beta\to \Gamma_*$ of ordered vector spaces over $\k$ that is the identity on $\Gamma$ and sends $\beta$ to $\beta_*$.
We define recursively $\beta_{*n}\in (\Gamma_*)_\infty$  by
$$\beta_{*0}:= \beta_*-\alpha_0, \qquad \beta_{*(n+1)} := \beta_{*n}^\dagger-\alpha_{n+1}.$$
Assume inductively that $\beta_{*0},\dots,\beta_{*n}\in \Gamma_{*}$ and that we have an embedding
$$e_n\ :\  \Gamma+\k\beta_0+ \cdots + \k\beta_n \to \Gamma_{*}$$
of ordered vector spaces over $\k$ that is the identity on
$\Gamma$ and sends $\beta_i$ to $\beta_{*i}$ for~$i=0,\dots,n$.
Then $\beta_n$ and $\beta_{*n}$ realize the same cut in $\Gamma$,
and so $\beta_{*n}^\dagger\notin \Gamma$, and $\beta_n^\dagger$ and~$\beta_{*n}^\dagger$ realize the same cut in 
$\Gamma$ by Lemma~\ref{cutinheritance}. Hence $\beta_{n+1}$ and~$\beta_{*(n+1)}\in \Gamma_*\setminus \Gamma$
realize the same cut in $\Gamma$. Moreover,
$\beta^\dagger_{*n} < \beta_{*(n+1)}^\dagger$ by Lemma~\ref{b1b2}.  
We have
$$[\Gamma+\k\beta_0+\cdots +\k\beta_n]_{\k}\ =\ [\Gamma]_{\k}\cup\big\{[\beta_0]_{\k},\dots,[\beta_n]_{\k}\big\},\quad [\beta_0]_{\k} >\cdots >[\beta_n]_{\k} > [\beta_{n+1}]_{\k}.
$$
Let $D$ be the cut realized by $[\beta_{n+1}]_{\k}$ in 
$[\Gamma+\k\beta_0+\cdots +\k\beta_n]_{\k}$. Then the above together with $[\beta_{*n}]_{\k} > [\beta_{*(n+1)}]_{\k}$ shows 
that $[\beta_{*(n+1)}]_{\k}$ realizes the $e_n$-image of the cut $D$ 
in $\big[e_n(\Gamma+\k\beta_0+\cdots +\k\beta_n)\big]_{\k}$. 
Hence $e_n$ extends to an embedding
$$e_{n+1}\ :\  \Gamma+\k\beta_1+ \cdots + \k\beta_n+\k\beta_{n+1} \to \Gamma_{*}$$
of ordered vector spaces over $\k$ that is the identity on
$\Gamma$ and sends $\beta_{n+1}$ to $\beta_{*(n+1)}$.
This leads to a map $e\colon \Gamma\<\beta\>\to \Gamma_*$
that extends each $e_n$, and is therefore an embedding
of $H$-couples over $\k$. Since $P_{\beta}$ is the only
$H$-cut in $\Gamma\<\beta\>$, $e$ embeds~$\big(\Gamma\<\beta\>, \psi_{\beta}, P_{\beta}\big)$ into $(\Gamma_*, \psi_*, P_*)$ over $\Gamma$.
\end{proof}

\noindent
This concludes the proof of Theorem~\ref{hclQE}.  \qed

\medskip
\noindent
Let $T_{\k}$ be the $\mathcal{L}_{\k}$-theory of closed
$H$-triples over $\k$. Let $T_{\k}^{>}$ be the $\mathcal{L}_{\k}$-theory whose models are the closed $H$-triples $(\Gamma,\psi, P)$ over $\k$
with $0\in P$, equivalently~$\Psi\cap \Gamma^{>}\ne \emptyset$.  Let $T_{\k}^{<}$ be the $\mathcal{L}_{\k}$-theory whose models are the closed $H$-triples $(\Gamma,\psi, P)$ over $\k$
with $0\notin P$, equivalently $\Psi\subseteq \Gamma^{<}$. 

\begin{cor} The $\mathcal{L}_{\k}$-theory $T_{\k}$ has exactly two
completions: $T_{\k}^{>}$ and $T_{\k}^{<}$.
\end{cor}
\begin{proof} We have an $H$-triple $\big(\{0\},\psi_0, \{0\}\big)$ over 
$\k$ that embeds into every model of~$T_{\k}^{>}$, and an $H$-triple $\big(\{0\},\psi_0, \emptyset\big)$ over $\k$ that embeds into every model of $T_{\k}^{<}$. Here~$\psi_0$ is the 
``empty'' function $\emptyset \to \{0\}$. 
\end{proof}

\noindent
Suppose $K$ is a Liouville closed $H$-field. Then its $H$-couple $(\Gamma,\psi)$ is naturally an $H$-couple over its constant field $C$. The case $(\Gamma,\psi)\models T_{C}^{>}$ corresponds to the derivation $\der$ of $K$ being small (that is, $\der f \prec 1$ for all $f\prec 1$ in $K$),
while the case~$(\Gamma,\psi)\models T_{C}^{<}$ corresponds to this derivation not being small. For example, the usual derivation $\frac{d}{dx}$ of $\T$ is small. The derivation $x^2\frac{d}{dx}$ on $\T$ is not small, but $\T$ with this derivation is still Liouville closed.

\section{Simple Extensions}\label{sec:simple}

\noindent
Let $(\Gamma, \psi)$ be an $H$-couple over $\k$ with asymptotic integration, and let 
$(\Gamma^*, \psi^*)$ be an
$H$-couple over $\k$ that extends $(\Gamma, \psi)$. For $\gamma\in \Gamma^*$, let $\big(\Gamma\<\gamma\>,\psi_\gamma\big)$  denote the $H$-couple over $\k$
generated by $\Gamma\cup\{\gamma\}$ in $(\Gamma^*, \psi^*)$. Let $\beta\in \Gamma^*\setminus \Gamma$. The following result yields a
useful
description of the ``simple'' extension $\big(\Gamma\<\beta\>,\psi_\beta\big)$, where $i$,~$n$ range over $\N=\{0,1,2,\dots\}$:

\begin{prop}\label{cases} One of the following occurs:
\begin{enumerate}
\item[(a)] $(\Gamma+ \k\beta)^\dagger=\Gamma^\dagger$;
\item[(b)] there are sequences $(\alpha_i)$ in 
$\Gamma$ and $(\beta_i)$ in $\Gamma^*$ such that
$(\beta_i)$ is $\k$-linearly independent
over $\Gamma$, $\beta_0=\beta-\alpha_0$ and $\beta_{i+1}=\beta_i^\dagger-\alpha_{i+1}$ for all $i$, and such that
$\Gamma\<\beta\>=\Gamma \oplus \bigoplus_{i=0}^\infty \k\beta_i$.
\item[(c)$_n$] there are $\alpha_0,\dots,\alpha_n\in \Gamma$, and nonzero $\beta_0,\dots, \beta_n\in \Gamma^*$ such that $\beta_0=\beta-\alpha_0$, $\beta_{i+1}=\beta_i^\dagger-\alpha_{i+1}$ for $i<n$, the vectors $\beta_0,\dots, \beta_n, \beta_n^\dagger$ are $\k$-linearly independent over 
$\Gamma$, $(\Gamma+\k\beta_n^\dagger)^\dagger=\Gamma^\dagger$, and $\Gamma\<\beta\>=\Gamma \oplus \bigoplus_{i=0}^n \k\beta_i \oplus \k\beta_n^\dagger$.
\item[(d)$_n$] there are $\alpha_0,\dots,\alpha_n\in \Gamma$, and nonzero $\beta_0,\dots, \beta_n\in \Gamma^*$ such that $\beta_0=\beta-\alpha_0$, $\beta_{i+1}=\beta_i^\dagger-\alpha_{i+1}$ for $i<n$, the vectors $\beta_0,\dots, \beta_n$ are $\k$-linearly independent over 
$\Gamma$, $\beta_n^\dagger\in \Gamma\setminus \Gamma^\dagger$, and $\Gamma\<\beta\>=\Gamma \oplus \bigoplus_{i=0}^n \k\beta_i$.
\end{enumerate}
\end{prop}

\noindent
Note that in case (a) we have $\Gamma\<\beta\>=\Gamma\oplus \k\beta$,
a case described in more detail in Lemma~\ref{noextradaggers}.
The proof below gives extra information about the other cases.

\begin{proof}
Suppose we are not in case (a). Then we have $\alpha_0\in \Gamma$ and $\beta_0:=\beta-\alpha_0$ with $\beta_0^\dagger\notin \Gamma^\dagger$. 
This is the first step in inductively constructing elements $\alpha_i\in \Gamma$ and
$\beta_i\in \Gamma\<\beta\>\setminus \Gamma_0$, either for all $i$, or for all $i\le n$ for a certain $n$. 
Suppose we already have $\alpha_0,\dots, \alpha_n\in \Gamma$ and $\beta_0,\dots, \beta_n\in \Gamma\<\beta\>\setminus \Gamma$ with $\alpha_0$ and $\beta_0$ as above,  $\beta_{i+1}=\beta_i^\dagger-\alpha_{i+1}$ and $\beta_i^\dagger\notin \Gamma$ for $i < n$, and
$\beta_n^\dagger\notin \Gamma^\dagger$.
Thus $[\beta_i]_{\k}\notin [\Gamma]_{\k}$ for $i\le n$.   

\medskip\noindent
{\bf Claim~1}:\quad $\beta_0^\dagger < \cdots < \beta_n^\dagger$.

\medskip\noindent
{\bf Claim~2}:\quad there is no $\eta\in \Gamma + \k \beta_0+ \cdots + \k \beta_n$ with $ \Psi < \eta < (\Gamma^{>})'$.

\medskip\noindent
To prove Claim~1, assume towards a contradiction
that $\beta_i^\dagger\ge \beta_{i+1}^\dagger$, $i<n$. Then by
Lemma~\ref{b1b2} we have $0<|\beta_i|<\Gamma^{>}$, so 
$\Psi< \beta_i^\dagger < (\Gamma^{>})'$, and thus
$[\beta_{i+1}]_{\k}\in [\Gamma]_{\k}$ by Corollary~\ref{addgapcor}, a contradiction. 
It follows from Claim~1 that $[\beta_0]_{\k} > \cdots > [\beta_n]_{\k}$ and that $\beta_0,\dots, \beta_n$ are $\k$-linearly independent over $\Gamma$. As to Claim~2, suppose towards a contradiction that $\Psi < \gamma + \delta < (\Gamma^{>})'$ where $\gamma\in \Gamma$, $\delta\in \k\beta_0+\cdots + \k\beta_n$. 
Then~$\delta\ne 0$, and so $[\delta]_{\k}\notin [\Gamma]_{\k}$. With $D:= \Psi-\gamma$ and
$E:=(\Gamma^{>})'-\gamma$, we have~$D < \delta < E$. On the other hand, for every $\epsilon\in \Gamma^{>}$ there are $d\in D$ and $e\in E$ with $e-d<\epsilon$, so $\Gamma$ is dense in $\Gamma+\k\delta$ by \cite[2.4.17]{ADH}, contradicting $[\delta]_{\k}\notin [\Gamma]_{\k}$. This concludes the proof of Claim~2.   

\medskip\noindent
If $(\beta_n^\dagger-\alpha_{n+1})^\dagger\notin \Gamma^\dagger$ for some $\alpha_{n+1}\in \Gamma$ (so $\beta_n^\dagger\notin \Gamma$), then  we take such an $\alpha_{n+1}$ and set $\beta_{n+1}:=\beta_n^\dagger-\alpha_{n+1}$. If there is no such $\alpha_{n+1}$, then the construction breaks off, with
$\alpha_n$ and $\beta_n$ as the last vectors.

Suppose the construction goes on indefinitely. Then it yields infinite se\-quen\-ces~$(\alpha_i)$ and $(\beta_i)$
as in case (b), in particular, $\Gamma\<\beta\>=\Gamma \oplus \bigoplus_{i=0}^\infty \k\beta_i$, 
$$\Psi_{\beta}\ :=\ \psi^*\big(\Gamma\<\beta\>^{\ne}\big)\ =\ \Psi\cup\big\{\beta_i^\dagger:\,i\in \N\big\}, $$
and $\big(\Gamma\<\beta\>, \psi_\beta\big)$ has asymptotic integration by Claim~2.

\medskip\noindent
Next, assume that the construction stops with $\alpha_n$ and $\beta_n$ as the last vectors. Thus~$({\Gamma+\k\beta_n^\dagger})^\dagger=\Gamma^\dagger$. We have two cases:

\medskip\noindent
{\bf Case~1}: $\beta_n^\dagger\notin \Gamma$.  
Then
$\alpha_0,\dots, \alpha_n, \beta_0,\dots, \beta_n$ are as in case (c)$_n$. Here is why.  
Set~$\Delta:=\Gamma + \k\beta_n^\dagger$, so 
$\Delta^{\dagger}=\Gamma^\dagger$.
From $\beta_i^\dagger\notin \Delta^\dagger$ for all $i\le n$
and Claim~1 we obtain that $\beta_0,\dots, \beta_n$ are
$\k$-linearly independent over $\Delta$, with
$$(\Delta + \k\beta_0+\cdots + \k\beta_n)^\dagger\ \subseteq\
  \Delta + \k\beta_0+\cdots + \k\beta_n,$$ 
and $\beta\in \Delta + \k\beta_0$, which proves the assertion. 

\medskip\noindent
{\bf Case~2}: $\beta_n^\dagger\in\Gamma$. Then 
$\alpha_0,\dots, \alpha_n, \beta_0,\dots, \beta_n$ are as in case (d)$_n$. Here is why. From $\beta_i^\dagger\notin \Gamma^\dagger$ for all $i\le n$
and Claim~1 we obtain that $\beta_0,\dots, \beta_n$ are
$\k$-linearly independent over $\Gamma$, with
$$(\Gamma + \k\beta_0+\cdots + \k\beta_n)^\dagger\ \subseteq\
  \Gamma + \k\beta_0+\cdots + \k\beta_n,$$ 
and $\beta\in \Gamma + \k\beta_0$, which proves the assertion. 
\end{proof}

\noindent
In case (d)$_n$ we have $\beta_n^\dagger\in \Gamma\setminus \Gamma^\dagger$, and this cannot happen if $(\Gamma,\psi)$ is closed.
The proof of Proposition~\ref{cases} yields some further results that are needed later: 

\begin{lemma}\label{lemb} Let $(\alpha_i)$ and $(\beta_i)$ be as in ${\rm(b)}$. Then:
\begin{enumerate}
\item[(i)] $\beta_i^\dagger\notin \Gamma$ for all $i$, and thus $[\beta_i]_{\k}\notin [\Gamma]_{\k}$ for all $i$; 
\item[(ii)] $(\beta_i^\dagger)$ is strictly increasing, and thus $([\beta_i]_{\k})$ is strictly decreasing;
\item[(iii)] $\big[\Gamma\<\beta\>\big]_{\k}=[\Gamma]_{\k}\cup\big\{[\beta_i]_{\k}:\,i\in \N\big\}$, and thus $\Psi_{\beta} = \Psi\cup \big\{\beta_i^\dagger:\, i\in \N\big\}$;
\item[(iv)] there is no $\eta\in \Gamma\<\beta\>$ with $\Psi < \eta < (\Gamma^{>})'$;
\item[(v)] $\big(\Gamma\<\beta\>, \psi_\beta\big)$ has asymptotic integration;
\item[(vi)] $\Gamma^{<}$ is cofinal in $\Gamma\<\beta\>^{<}$.
\end{enumerate}
If
$(\Gamma, \psi)$ is closed and $\gamma\in \Gamma^*\setminus \Gamma$ realizes the same cut in $\Gamma$ as $\beta$, then we have an
isomorphism $\big(\Gamma\<\beta\>, \psi_{\beta}\big)\to \big(\Gamma\<\gamma\>, \psi_{\gamma}\big)$
of $H$-couples over $\k$ that is the identity on~$\Gamma$ and sends $\beta$ to $\gamma$.  If $(\Gamma, \psi)$ is of Hahn type, then so is $\big(\Gamma\<\beta\>, \psi_{\beta}\big)$. 
\end{lemma}
\begin{proof} As to (i), this follows from the $\k$-linear independence of $(\beta_i)$
over $\Gamma$ and from~$\beta_{i}^\dagger=\beta_{i+1}+\alpha_{i+1}$. Hence the sequences 
$(\alpha_i),$ and $(\beta_i)$ conform to the construction in the proof of Proposition~\ref{cases}, and so other parts of that proof
yield (ii)--(vi). 
The next statement follows as in the proof of
Lemma~\ref{cinf} using Lemma~\ref{cutinheritance} and (iv). 

Suppose that $(\Gamma,\psi)$ is of Hahn type.  We show that then $\Gamma\<\beta\>$ is a Hahn space; the additional argument required for showing that $\big(\Gamma\<\beta\>,\psi_{\beta}\big)$ is of Hahn type is similar and left to the reader. 
So let $\delta_1, \delta_2\in \Gamma\<\beta\>^{\ne}$ satisfy $[\delta_1]_{\k}=[\delta_2]_{\k}$;  
we have to find $c\in \k$ such that $[\delta_1-c\delta_2]_{\k}< [\delta_1]_{\k}$. 
Now $$\delta_1\ =\ \gamma_1+\sum_i c_{i1}\beta_i, \qquad \delta_2\ =\ \gamma_2+\sum_i c_{i2}\beta_i,\quad \gamma_1,\gamma_2\in \Gamma,$$
with all $c_{i1}, c_{i2}\in \k$, and $c_{i1}=c_{i2}=0$ for all but finitely many $i$. Consider first the case
$[\delta_1]_{\k}\in [\Gamma]_{\k}$. Then $[\gamma_1]_{\k}> [\beta_i]_{\k}$ for all $i$ with $c_{i1}\ne 0$, by~(i),~(ii),~(iii), and so~$\delta_1=\gamma_1+ \alpha_1$ with $[\alpha_1]_{\k}< [\gamma_1]_{\k}=[\delta_1]_{\k}$, and likewise $\delta_2=\gamma_2 + \alpha_2$
with~$[\alpha_2]_{\k}< [\gamma_2]_{\k}=[\delta_2]_{\k}$. Take $c\in \k$ such that
$[\gamma_1-c\gamma_2]_{\k}<[\gamma_1]_{\k}$. Then $\delta_1-c\delta_2= \gamma_1-c\gamma_2+ \alpha_1-c\alpha_2$, so
$[\delta_1-c\delta_2]_{\k}<[\gamma_1]_{\k}=[\delta_1]_{\k}$. Next, suppose~$[\delta_1]_{\k}\notin [\Gamma]_{\k}$. 
Then $c_{i1}\ne 0$ for some $i$; let $j$ be the least such $i$. Then~$[\gamma_1]_{\k}<[\beta_j]_{\k}$ and~$[\delta_1]_{\k}=[\beta_j]_{\k}$ by (ii).  Now~$j$ is also the least  $i$ with $c_{i2}\ne 0$, in view of~$[\delta_1]_{\k}=[\delta_2]_{\k}$.
Then~$[\delta_1-c\delta_2]_{\k}<[\delta_1]_{\k}$ for~$c\in \k$ with $c_{j1}=cc_{j2}$. 
\end{proof}

\begin{lemma}\label{lemcn1} Let $\alpha_0,\dots, \alpha_n, \beta_1,\dots, \beta_n$ be as in \textup{(c)$_n$}, and set $\Delta:=\Gamma+\k\beta_n^\dagger$, so~$\Delta^\dagger=\Gamma^\dagger$ and $\Gamma\<\beta\>=\Delta \oplus\k\beta_0\oplus \cdots \oplus \k\beta_n$. Then: 
\begin{enumerate}
\item[(i)] $\Gamma^{<}$ is cofinal in $\Delta^{<}$;
\item[(ii)] $\beta_0^\dagger, \dots, \beta_{n}^\dagger\notin \Gamma$, and thus 
$[\beta_0]_{\k},\dots,[\beta_n]_{\k}\notin [\Delta]_{\k}$;
\item[(iii)] $\beta_0^\dagger< \cdots < \beta_n^\dagger$, and thus 
$[\beta_0]_{\k}> \cdots > [\beta_n]_{\k}$;
\item[(iv)] $\Psi_{\beta} = \Psi\cup \{\beta_0^\dagger,\dots, \beta_n^\dagger\}$ and $\big[\Gamma\<\beta\>\big]_{\k}=[\Delta]_{\k}\cup\big\{[\beta_0]_{\k},\dots,[\beta_n]_{\k}\big\}$;  
\item[(v)] there is no $\gamma\in \Delta+ \k\beta_0+\cdots + \k\beta_{n-1}$ with $0 < \gamma < \Gamma^{>}$;
\item[(vi)] if $|\beta_n|\ge \alpha$ for some $\alpha\in \Gamma^{>}$, then $\Gamma^{<}$ is cofinal in $\Gamma\<\beta\>^{<}$ and so a gap in~$(\Delta, \psi_{\Delta})$, if any, remains a gap in~$\big(\Gamma\<\beta\>, \psi_{\beta}\big)$;
\item[(vii)] if $|\beta_n|< \Gamma^{>}$, then $\big(\Gamma\<\beta\>, \psi_{\beta}\big)$ is grounded with $\max \Psi_{\beta}=\beta_n^\dagger$;
\item[(viii)] if $(\Delta, \psi_{\Delta})$ has no gap, then there is no $\eta\in \Gamma\<\beta\>$ with $\Psi < \eta < (\Gamma^{>})'$, and so $\Gamma^{<}$ is cofinal in $\Gamma\<\beta\>^{<}$  and $\big(\Gamma\<\beta\>, \psi_{\beta}\big)$ has asymptotic integration. 
\end{enumerate}
 \end{lemma}
\begin{proof} As to (i), if $\delta\in \Delta$ and
$\Gamma^{<}< \delta < 0$, then $\Psi < \delta^\dagger$, contradicting $\Delta^\dagger=\Gamma^\dagger$.  
Item (ii) follows from the $\k$-linear independence of $\beta_0,\dots, \beta_n, \beta_n^\dagger$
over $\Gamma$ and from~$\beta_{i}^\dagger=\beta_{i+1}+\alpha_{i+1}$ for $ i < n$. Next we obtain (iii) from Claim~1 in the proof of Proposition~\ref{cases}, and then (iv) follows easily. 
As to (v), by (ii) and (iii) we have
$$[\Delta+ \k\beta_0+\cdots + \k\beta_{n-1}]_{\k}\ =\ [\Delta]_{\k}\cup\big\{[\beta_0]_{\k},\dots, [\beta_{n-1}]_{\k}\big\}.$$
Thus assuming towards a contradiction that (v) is false gives
$\gamma\in \Delta\cup\{\beta_0,\dots, \beta_{n-1}\}$
with $0 < |\gamma| < \Gamma^{>}$. Then $\Psi<\gamma^\dagger < (\Gamma^{>})'$, and so $\gamma\notin \Delta$. Hence $\gamma=\beta_i$ with $ i < n$, and so $\gamma^\dagger\in \Gamma+\k\beta_0+ \cdots + \k\beta_n$, contradicting Claim~2 in the proof of
Proposition~\ref{cases} with $\gamma^\dagger$ in the role of $\eta$. By similar arguments, if $0 < \gamma < \Gamma^{>}$
for some
$\gamma\in \Gamma\<\beta\>$, then $0 < |\beta_n| < \Gamma^{>}$.
This gives (vi). For (vii), assume
$|\beta_n|< \Gamma^{>}$. Then~(i),~(iv),~(v) give
$[\beta_n]_{\k}=\min\!\big[\Gamma\<\beta\>^{\ne}\big]_{\k}$, and thus  $\max \Psi_\beta=\beta_n^\dagger$.

As to (viii), note first that $\Psi=\Psi_{\Delta}$. Assume 
$(\Delta, \psi_{\Delta})$ has no gap. Then 
$(\Delta, \psi_{\Delta})$
has asymptotic integration. Hence by Claim~2 in the proof of Proposition~\ref{cases}, applied to $\Delta$ instead of $\Gamma$,
there is no $\eta\in \Gamma\<\beta\>$ with $\Psi < \eta < (\Gamma^{>})'$.    
\end{proof}

\begin{lemma}\label{lemdn} Let $\alpha_0,\dots, \alpha_n, \beta_0,\dots, \beta_n$ be as in (d)$_n$. Then: 
\begin{enumerate}
\item[(i)] $\beta_0^\dagger, \dots, \beta_{n-1}^\dagger\notin \Gamma$, $\beta_n^\dagger\notin \Psi$, and thus 
$[\beta_0]_{\k},\dots,[\beta_n]_{\k}\notin [\Gamma]_{\k}$;
\item[(ii)] $\beta_0^\dagger< \cdots < \beta_n^\dagger$, and thus 
$[\beta_0]_{\k}> \cdots > [\beta_n]_{\k}$;
\item[(iii)] $\Psi_{\beta} = \Psi\cup \{\beta_0^\dagger,\dots, \beta_n^\dagger\}$ and $\big[\Gamma\<\beta\>\big]_{\k}=[\Gamma]_{\k}\cup\big\{[\beta_0]_{\k},\dots,[\beta_n]_{\k}\big\}$; 
\item[(iv)] there is no $\eta\in \Gamma\<\beta\>$ with $\Psi < \eta < (\Gamma^{>})'$;
\item[(v)] $\Gamma^{<}$ is cofinal in $\Gamma\<\beta\>^{<}$, and
$\big(\Gamma\<\beta\>, \psi_{\beta}\big)$ has asymptotic integration.
\end{enumerate} 
\end{lemma}
\begin{proof} The first part of (i) follows from the
recursion satisfied by $\beta_0,\dots, \beta_n$, the
$\k$-linear independence of $\beta_0,\dots, \beta_n$ over $\Gamma$, and $\beta_n^\dagger\notin \Psi$. Claim~1 in the proof of Proposition~\ref{cases} gives (ii), which together with (i) yields (iii).
Claim~2 in that proof gives (iv), which has (v) as an easy consequence. 
\end{proof} 

\noindent
The next result is crucial in the proof  of Theorem~\ref{thm} in Section\ref{chha}. Here $(\Gamma^*, \psi^*)$ is equipped with an $H$-cut
$P^*$, and we set $P:= P^*\cap \Gamma = \Psi^{\downarrow}$, and $P_{\gamma}:=P^*\cap \Gamma\<\gamma\>$ for $\gamma\in \Gamma^*$, so we have the $H$-triples $(\Gamma, \psi, P), \big(\Gamma\<\gamma\>, \psi_{\gamma}, P_{\gamma}\big) \subseteq (\Gamma^*, \psi^*, P^*)$ over $\k$.

\begin{lemma}\label{keyinterval} Assume 
$(\Gamma^*,\psi^*)$ is closed, of Hahn type, and
$\Gamma^{<}$ is not cofinal in~$(\Gamma^*)^{<}$. Then for some
$\delta\in (\Gamma^*)^{>}$, all $\gamma\in \Gamma^*$
with $|\beta-\gamma| < \delta$ yield an isomorphism
$\big(\Gamma\<\beta\>, \psi_{\beta}, P_{\beta}\big) \to
\big(\Gamma\<\gamma\>, \psi_{\gamma}, P_{\gamma}\big)$ over
$\Gamma$ sending $\beta$ to $\gamma$. 
\end{lemma}
\begin{proof} Suppose we are in Case~(a) of Proposition~\ref{cases}. There are three subcases:

\medskip\noindent
{\em Subcase 1}: $(\Gamma^>)^\dagger < \eta < (\Gamma^>)'$
and $\eta\in P^*$ for some $\eta\in \Gamma+\k\beta$. 
Fix such $\eta$ and recall from Case~1 in the proof of Lemma~\ref{noextradaggers} that $\Gamma$ is dense in $\Gamma+\k\eta=\Gamma+\k\beta$. Thus if $\varepsilon\in \Gamma^*$ and
$0 < \varepsilon < \Gamma^{>}$, then $(\Gamma^>)^\dagger < \eta-\varepsilon < \eta$. Moreover, $P^*$ has no largest element,
so we can take $\varepsilon\in (\Gamma^*)^{>}$ so small that for all 
$\zeta\in \Gamma^*$ with~$|\eta-\zeta|<\varepsilon$ we have
$(\Gamma^>)^\dagger < \zeta < (\Gamma^>)'$
and $\zeta\in P^*$; in particular, such $\zeta$
realizes the same cut in $\Gamma$ as $\eta$. 
Take $\alpha\in \Gamma$ and $c\in \k^\times$ with $\beta=\alpha+c\eta$.  Then for
$\zeta$ as above
and $\gamma:= \alpha+c\zeta$ the condition
$|\eta-\zeta|< \varepsilon$ amounts to 
$|\beta-\gamma|< \delta:=|c|\varepsilon$, with an isomorphism
$\big(\Gamma\<\beta\>, \psi_{\beta}, P_{\beta}\big) \to
\big(\Gamma\<\gamma\>, \psi_{\gamma}, P_{\gamma}\big)$ over
$\Gamma$ sending $\beta$ to $\gamma$.

\medskip\noindent
{\em Subcase 2}: $(\Gamma^>)^\dagger < \eta < (\Gamma^>)'$
and $\eta\notin P^*$ for some $\eta\in \Gamma+\k\beta$.
This can be treated in the same way as Subcase 1. 

\medskip\noindent
{\em Subcase 3}: there is no $\eta\in \Gamma+\k\beta$ with $(\Gamma^>)^\dagger < \eta < (\Gamma^>)'$. Take $\delta\in \Gamma^*$ such that~$0 < \delta < \Gamma^{>}$. Then all $\gamma\in \Gamma^*$ with
$|\gamma-\beta|< \delta$ realize the same cut in $\Gamma$ as $\beta$: otherwise we would have $\alpha\in \Gamma$ with
$0<|\alpha-\beta|< \Gamma^{>}$, so $(\Gamma^>)^\dagger < (\alpha-\beta)^\dagger < (\Gamma^>)'$, a contradiction. Now $(\Gamma^*, \psi^*)$ is of Hahn type, so $[\Gamma+\k\beta]_{\k}=[\Gamma]_{\k}$. As in Case~3 in the proof of Lemma~\ref{noextradaggers} this yields for any such $\gamma$ an isomorphism 
$\big(\Gamma\<\beta\>, \psi_{\beta}, P_{\beta}\big) \to 
\big(\Gamma\<\gamma\>, \psi_{\gamma}, P_{\gamma}\big)$ over
$\Gamma$ sending $\beta$ to $\gamma$. 

\medskip\noindent
Assume we are in Case~(b) of Proposition~\ref{cases}, and let $(\alpha_i)$ and $(\beta_i)$ be as in that
case. Let $\varepsilon\in \Gamma^*$ be such that $[\varepsilon]_{\k}< [\beta_0]_{\k}$. Then $\beta_0+\varepsilon= (\beta+\varepsilon)-\alpha_0$, $[\beta_0+\varepsilon]_{\k}=[\beta_0]_{\k}$, and thus
$(\beta_0+\varepsilon)^\dagger=\beta_0^\dagger$. It follows that
with $\beta+\varepsilon$ instead of $\beta$ we are
also in case~(b), with associated sequences $(\alpha_i)$ and
$(\beta_{i,\varepsilon})$, with $\beta_{0,\varepsilon}:=\beta_0+\varepsilon$ and~$\beta_{i,\varepsilon}:= \beta_i$ for~$i\ge 1$. As noted in the proof of Lemma~\ref{lemb}, the sequences~$(\alpha_i)$,~$(\beta_i)$ conform to the construction in the proof of Proposition~\ref{cases}, and so the latter proof yields an isomorphism $\big(\Gamma\<\beta\>, \psi_{\beta}, P_{\beta}\big) \to 
\big(\Gamma\<\beta+\varepsilon\>, \psi_{\beta+\varepsilon}, P_{\beta+\varepsilon}\big)$ over
$\Gamma$ that sends~$\beta_i$ to~$\beta_{i,\varepsilon}$
for each $i$, and thus $\beta$ to $\beta+\varepsilon$.

\medskip\noindent
Next, assume we are in Case~(c)$_n$ of Proposition~\ref{cases}, and let $\alpha_0,\dots, \alpha_n,\beta_0,\dots, \beta_n$ be as in that case. As before, let $\varepsilon\in \Gamma^*$ be such that $[\varepsilon]_{\k}< [\beta_0]_{\k}$. Then $\beta_0+\varepsilon= (\beta+\varepsilon)-\alpha_0$, $[\beta_0+\varepsilon]_{\k}=[\beta_0]_{\k}$, so
$(\beta_0+\varepsilon)^\dagger=\beta_0^\dagger$. Hence
with $\beta+\varepsilon$ instead of $\beta$ we are
again in case (c)$_n$, with associated sequences $\alpha_0,\dots, \alpha_n$ and
$\beta_{0,\varepsilon}, \dots, \beta_{n,\varepsilon}$, with $\beta_{0,\varepsilon}:=\beta_0+\varepsilon$ and 
$\beta_{i,\varepsilon}:= \beta_i$ for $1\le i\le n$. Note also that
$\beta$ and $\beta+\varepsilon$ give rise to the same
$\Delta=\Gamma + \k\beta_n^\dagger=\Gamma + \k \beta_{n,\varepsilon}^\dagger$. 
It now follows from Lemma~\ref{lemcn1} that we have an isomorphism $\big(\Gamma\<\beta\>, \psi_{\beta}\big) \to
\big(\Gamma\<\beta+\varepsilon\>, \psi_{\beta+\varepsilon}\big)$
of $H$-couples over $\k$ that is the identity on
$\Delta$ and sends $\beta_i$ to $\beta_{i,\varepsilon}$
for each $i\le n$, and thus $\beta$ to $\beta+\varepsilon$. 
Since~$\beta$ and $\beta+\varepsilon$ yield the same $\Delta$, it follows easily from~(vi),~(vii),~(viii) of Lemma~\ref{lemcn1}
that this isomorphism maps $P_{\beta}$ onto $P_{\beta+\varepsilon}$.

\medskip\noindent
Finally, assume we are in Case~(d)$_n$ of Proposition~\ref{cases}, and let $\alpha_0,\dots, \alpha_n,\beta_0,\dots, \beta_n$ be as in that case. Let $\varepsilon\in \Gamma^*$ be such that $[\varepsilon]_{\k}< [\beta_0]_{\k}$. Then $\beta_0+\varepsilon= {(\beta+\varepsilon)-\alpha_0}$, $[\beta_0+\varepsilon]_{\k}=[\beta_0]_{\k}$, so
$(\beta_0+\varepsilon)^\dagger=\beta_0^\dagger$. Hence
with $\beta+\varepsilon$ instead of $\beta$ we are
again in case~(d)$_n$, with associated sequences $\alpha_0,\dots, \alpha_n$ and
$\beta_{0,\varepsilon}, \dots, \beta_{n,\varepsilon}$, with~$\beta_{0,\varepsilon}:=\beta_0+\varepsilon$ and $\beta_{i,\varepsilon}:= \beta_i$ for $1\le i\le n$. Then Lemma~\ref{lemdn} yields an isomorphism~$\big(\Gamma\<\beta\>, \psi_{\beta}, P_\beta\big) \to
\big(\Gamma\<\beta+\varepsilon\>, \psi_{\beta+\varepsilon}, P_{\beta+\varepsilon}\big)$
of $H$-triples over $\k$ that is the identity on
$\Gamma$ and sends $\beta_i$ to~$\beta_{i,\varepsilon}$
for each $i\le n$, and thus $\beta$ to $\beta+\varepsilon$.  
\end{proof}

\section{Closed $H$-couples of Hahn Type}\label{chha} 

\noindent
So far we have treated $H$-couples over $\k$ as one-sorted
structures, by keeping $\k$ fixed and having for each scalar $c$ a separate unary function symbol that is interpreted as scalar multiplication by 
$c$. We now go to the setting where an $H$-couple over $\k$ is viewed as a $2$-sorted structure with $\k$ as a second sort, and thus with ``Hahn type'' as a first-order condition. Extending an $H$-couple may now involve extending $\k$, so we begin with a subsection on the process of scalar extension for Hahn spaces. We remind the reader that the ordered scalar field $\k$ is not
necessarily real closed.

\subsection*{Scalar extension}
Let $\Gamma$ be a Hahn space over $\k$, and let $\k^*$ be an ordered field extension of $\k$. Then we
have the vector space $\Gamma_{\k^*}:= \k^*\otimes_{\k} \Gamma$ over $\k^*$.
We have the $\k$-linear embedding 
$\gamma\mapsto 1\otimes \gamma\colon \Gamma\to \Gamma_{\k^*}$ via which we identify $\Gamma$ with a $\k$-linear subspace of 
$\Gamma_{\k^*}$.
We make $\Gamma_{\k^*}$ into a Hahn space over
$\k^*$ as follows: for any 
$\gamma\in\Gamma_{\k^*}^{\ne}$
we have $\gamma=c_1\gamma_1+\cdots + c_m\gamma_m$
with $m\ge 1$, $c_1,\dots, c_m\in (\k^*)^\times$, $\gamma_1\dots, \gamma_m\in \Gamma^{>}$, $[\gamma_1]_{\k} > \cdots >[\gamma_m]_{\k}$;
then $\gamma>0$ iff $c_1>0$. This makes $\Gamma$ into an
ordered $\k$-linear subspace of $\Gamma_{\k^*}$, and we have an
order-preserving bijection~${[\gamma]_{\k}\to [\gamma]_{\k^*}\colon [\Gamma]_{\k} \to [\Gamma_{\k^*}]_{\k^*}}$. 

\begin{lemma}\label{ndha} Assume 
$[\Gamma^{\ne}]_{\k}$ has no least element. Then for every
$\gamma^{*}\in \Gamma_{\k^{*}}\setminus \Gamma$ there is an element 
$\varepsilon\in \Gamma^{>}$ such that $|\gamma^{*}-\gamma|>\varepsilon$ for all $\gamma\in \Gamma$.
\end{lemma}
\begin{proof} Let $\gamma^*\in \Gamma_{\k^*}\setminus \Gamma$, so $\gamma^*=c_1\gamma_1+\cdots + c_m\gamma_m$
with $m\ge 1$, $c_1,\dots, c_m\in (\k^*)^\times$, $\gamma_1\dots, \gamma_m\in \Gamma^{>}$, $[\gamma_1]_{\k} > \cdots >[\gamma_m]_{\k}$. To show that $\gamma^*$ has the claimed property we can assume $c_1\notin \k$. Take any $\varepsilon\in \Gamma^{>}$ with $[\varepsilon]_{\k} < [\gamma_1]_{\k}$, and assume towards a contradiction that $\gamma\in \Gamma$ and
$|\gamma^*-\gamma|\le \varepsilon$. Then $[\gamma]_{\k^*}=[\gamma^*]_{\k^*}=[\gamma_1]_{\k^*}$, so~$[\gamma]_{\k}=[\gamma_1]_{\k}$, and hence $[\gamma-c\gamma_1]_{\k}<[\gamma_1]_{\k}$ with $c\in \k$. In view of
$$\gamma^*-\gamma\ =\ (c_1-c)\gamma_1 + c_2\gamma_2 + \cdots + c_m\gamma_m - (\gamma-c\gamma_1)$$
and $c_1\ne c$, this yields a contradiction.  
\end{proof}

\noindent
We also have the following universal property:

\begin{cor}\label{ascha} Any embedding $\Gamma \to \Gamma^*$ of ordered vector spaces over $\k$ into an ordered vector space $\Gamma^*$ over 
$\k^*$ such that the induced map $[\Gamma]_{\k} \to [\Gamma^*]_{\k^*}$ is injective extends uniquely to an embedding $\Gamma_{\k^*} \to \Gamma^*$ of ordered vector spaces over $\k^*$.
\end{cor}

\noindent
Let $(\Gamma,\psi)$ be an $H$-couple over $\k$ of Hahn type and 
$\k^*$ an ordered field extension of~$\k$. The
$H$-couple $(\Gamma, \psi)_{\k^*}:= (\Gamma_{\k^*}, \psi_{\k^*})$ over $\k^*$ is
determined by requiring that~$\psi_{\k^*}$ extends $\psi$. 
Note that then $(\Gamma, \psi)_{\k^*}$
is also of Hahn type and has the same $\Psi$-set as~$(\Gamma,\psi)$. The following is close to \cite[Lemma 3.7]{AvdD}, whose proof uses a form of Hahn's Embedding Theorem. Here we use instead Lemma~\ref{ndha}. 

\begin{lemma}\label{scpr} If $\gamma\in \Gamma$ is a gap in $(\Gamma,\psi)$,
then $\gamma$ remains a gap in $(\Gamma,\psi)_{\k^*}$.
If $\gamma^*$ is a gap in $(\Gamma,\psi)_{\k^*}$, then $\gamma^*\in \Gamma$. Thus
$(\Gamma, \psi)$ has asymptotic integration if and only if 
$(\Gamma, \psi)_{\k^*}$ has asymptotic integration.
\end{lemma}
\begin{proof}  Suppose towards a contradiction that 
$\gamma\in \Gamma$ is a gap in $(\Gamma,\psi)$, but not in~$(\Gamma,\psi)_{\k^*}$. Then $\gamma=\alpha'$ with $\alpha\in \Gamma_{\k^*}^{>}\setminus \Gamma$. From $\gamma< (\Gamma^{>})'$ we get 
$0<\alpha<\Gamma^{>}$, but this
contradicts that by Lemma~\ref{ndha} 
we have $|\alpha |> \varepsilon$ for some $\varepsilon\in \Gamma^{>}$. 

Next, assume $\gamma^*$ is a gap in $(\Gamma,\psi)_{\k^*}$.
Then $\Psi < \gamma^* < (\Gamma^{>})'$, and for all 
$\varepsilon\in \Gamma^{>}$ there are $\alpha\in \Psi$ and
$\beta\in (\Gamma^{>})'$ (namely $\alpha:= \varepsilon^\dagger$
and $\beta:= \varepsilon'$) with $\beta-\alpha \le \varepsilon$.
In view of Lemma~\ref{ndha} this yields $\gamma^*\in \Gamma$. 
\end{proof}

\subsection*{Normalized $H$-couples} Let 
$(\Gamma, \psi)$ be an $H$-couple over $\k$. By \cite[Section~9.2]{ADH}, if~$\Psi\cap \Gamma^{>}\ne \emptyset$, then $\psi(\gamma)=\gamma$ for a unique $\gamma\in \Gamma^{>}$; this unique fixed point of 
$\psi$ on $\Gamma^{>}$ is then denoted by $1$. Referring to
$(\Gamma,\psi)$ as a {\em normalized $H$-couple\/} means that~${\Psi\cap \Gamma^{>}\ne \emptyset}$, and that
we consider $\Gamma$ as equipped with this fixed point $1$ as a distinguished element.  
(The term ``normalized'' is justified, because for any $H$-couple
over $\k$ with underlying ordered vector space $\Gamma\ne \{0\}$ we can arrange~${\Psi\cap \Gamma^{>}\ne \emptyset}$ by replacing its function $\psi$ with a suitable ``shift'' $\alpha+\psi$ where $\alpha\in \Gamma$.) 
For minor technical reasons it is convenient to restrict our attention in the remainder of this paper to 
normalized $H$-couples; this is hardly a loss of generality, as we saw. Note also that the $H$-couple of $\T$ is normalized 
by taking $1=v(x^{-1})$.

\medskip\noindent
Below we construe a normalized $H$-couple over $\k$ as a $2$-sorted structure
$$\Ga\ =\ \big((\Gamma,\psi), \k; \sc\!\big)$$ where $(\Gamma,\psi)$ is an 
$H$-couple as defined in the beginning of Section~\ref{prelim},
$\k$ is an ordered field,
and $\sc\colon \k\times \Gamma \to \Gamma$ is a scalar multiplication
that makes $\Gamma$ into an ordered vector space over $\k$
(but we shall write $c\gamma$ instead of $\sc(c,\gamma)$ for $c\in \k$ and $\gamma\in \Gamma$), such that $\psi(c\gamma)=\psi(\gamma)$
for $c\in \k^\times$, $\gamma\in \Gamma$; in addition we assume $\Gamma$ to be equipped with an element $1>0$ such that $\psi(1)=1$. Such $\Ga$ is said to be  {\em of Hahn type\/} if 
the $H$-couple $(\Gamma,\psi)$ over $\k$ is of Hahn type as
defined in Section~\ref{prelim}. 
In the same way we may consider a normalized $H$-triple over $\k$ as a $2$-sorted structure
$$\Ga=\big((\Gamma,\psi,P), \k; \sc\!\big).$$

\subsection*{The language and theory of normalized $H$-triples of Hahn type} We 
construe a normalized $H$-triple $\Ga=\big((\Gamma,\psi,P), \k; \sc\!\big)$ of Hahn type as an 
$\mathcal{L}_H$-structure, where $\mathcal{L}_H$ is the two-sorted language with the following non-logical symbols: \begin{enumerate}
\item[(i)] $P$, $<$, $0$, $1$, $\infty$, $-$, $+$, $\psi$, interpreted as usual in 
$\Gamma_{\infty}:=\Gamma\cup\{\infty\}$, the linear ordering on $\Gamma$ being extended to a linear order on $\Gamma_{\infty}$
by $\gamma<\infty$ for $\gamma\in \Gamma$, and
 with $\infty$ serving as a default value by setting $-\infty=\infty$,
$\gamma+\infty=\infty+\gamma=\infty+\infty=\psi(0)=\psi(\infty)=\infty$ for $\gamma\in \Gamma$;
\item[(ii)] $<$, $0$, $1$, $\infty$, $-$, $+$, $\,\cdot\,$, interpreted as usual  in 
$\k_{\infty}:=\k\cup\{\infty\}$, the linear ordering on $\k$ being extended to a linear order on $\k_{\infty}$
by $c<\infty$ for $c\in \k$, and
 with $\infty$ serving as a default value by setting $-\infty=\infty$,
$c+\infty=\infty+c=\infty+\infty=c\infty=\infty c= \infty\infty=\infty$ for $c\in \k$;
\item[(iii)] a symbol $\sc$ for the map $\k_{\infty}\times \Gamma_{\infty} \to \Gamma_{\infty}$ that is the scalar multiplication on $\k\times \Gamma$, and taking the value $\infty$ at all other points of $\k_{\infty}\times \Gamma_{\infty}$;
\item[(iv)]  a symbol $:$ for the function $\Gamma_{\infty}^2\to \k_{\infty}$ that assigns to every $(\alpha,\beta)\in \Gamma^2$ with~$[\alpha]_{\k}\le [\beta]_{\k}$ and $\beta\ne 0$ the unique scalar
$\alpha:\beta=c\in \k$ such that~$[{\alpha-c\beta}]_{\k}<[\beta]_{\k}$, and assigns to all other pairs in 
$\Gamma_{\infty}^2$ the value $\infty$.
\end{enumerate}

\noindent
The symbols in (i) should be distinguished from those in (ii) even though we use the same written signs for convenience. The two default values $\infty$ are included to make all primitives totally defined. Note that in (iv) we have $\alpha:\beta=0$ if $[\alpha]_{\k} < [\beta]_{\k}$.

Using $a1:b1=a/b$ for 
$a,b\in \k$ with $b\ne 0$, we see that a substructure of a normalized $H$-triple of Hahn type is also a normalized $H$-triple of Hahn type, with possibly smaller scalar field. 
  Thus the $\mathcal{L}_H$-theory of normalized $H$-triples
  of Hahn type has a universal axiomatization (which would be easy to specify). Let there be given normalized $H$-triples of Hahn type,
  $$ \Ga_0\ =\ \big((\Gamma_0,\psi_0,P_0), \k_0; \sc_0\!\big)\ \text{ and }\ 
\Ga\ =\ \big((\Gamma,\psi,P), \k; \sc\!\big).$$
An {\em embedding\/}
$\Ga_0 \to \Ga$ is a pair $i=(i_{\operatorname{v}}, i_{\operatorname{s}})$ whose vector part
$i_{\operatorname{v}}\colon \Gamma_0\to \Gamma$ is an embedding of ordered abelian 
group and whose scalar part $i_{\operatorname{s}}\colon \k_0 \to \k$ is an embedding of
ordered fields such that $i_{\operatorname{v}}(c\gamma)=i_{\operatorname{s}}(c)i_{\operatorname{v}}(\gamma)$ and
$\gamma\in P_0\Leftrightarrow i_{\operatorname{v}}(\gamma)\in P$  for all
$c\in \k_0$ and $\gamma\in \Gamma_0$, and  $i_{\operatorname{v}}\big(\psi_0(\gamma)\big)=\psi\big(i_{\operatorname{v}}(\gamma)\big)$ for all nonzero $\gamma\in \Gamma_0$ (and
so $i_{\operatorname{v}}(1)=1$ and $i_c(\alpha:\beta)=i_{\operatorname{v}}(\alpha):i_{\operatorname{v}}(\beta)$ for all $\alpha,\beta\in \Gamma$). If $\k_0=\k$, then an embedding
$e\colon(\Gamma_0, \psi_0, P_0) \to (\Gamma,\psi,P)$ 
of $H$-triples over $\k$ in the usual sense yields an embedding
$(e,\text{id}_{\k})\colon \Ga_0 \to \Ga$ as above.

\subsection*{Quantifier elimination} Let $T_H$ be the 
$\mathcal{L}_H$-theory of normalized closed
$H$-triples of Hahn type, and recall that the $H$-couple of $\T$ is naturally a model of $T_H$. In this subsection we let
$\Ga = \big((\Gamma,\psi,P), \k; \sc\!\big)$  and 
$\Ga^* = \big((\Gamma^*,\psi^*,P^*), \k^*; \sc^*\!\big)$ 
denote 
{\em normalized closed 
$H$-triples of Hahn type}, construed as models of $T_H$. 
The key embedding result is as follows:

\begin{prop}\label{propeq} Assume $\Ga^*$ is $\kappa$-saturated for $\kappa=|\Gamma|^+$. Let $\Ga_0$ be a substructure of $\Ga$ with scalar field $\k_0$. Let an embedding $i_0\colon \Ga_0\to \Ga^*$ be given, and an embedding~$e\colon \k \to \k^*$ of ordered fields such that $e|_{\k_0}=(i_0)_{\operatorname{s}}$. Then $i_0$ can be extended to an embedding $i\colon \Ga\to \Ga^*$ such that $i_{\operatorname{s}}=e$.
\end{prop} 
\begin{proof} By Corollary~\ref{ascha} on extending scalars,
the remarks following it, and (to handle the $P$-predicate) Lemma~\ref{scpr} we can reduce to the case $\k_0=\k$. It remains to appeal to
the embedding result established in the proof of Theorem~\ref{hclQE}.   
\end{proof}

\noindent
In what follows, {\em formula\/} means {\em $\mathcal{L}_H$-formula}. Let $x=(x_1,\dots, x_m)$ denote a tuple of distinct
scalar variables and $y=(y_1,\dots,y_n)$ a tuple of distinct vector variables. 

\begin{cor}\label{coreq} Suppose that $\Ga$ is a substructure of
$\Ga^*$. Then
$$\Ga\preceq \Ga^* \text{ \textup{(}as $\mathcal{L}_H$-structures\textup{)}}\quad \Longleftrightarrow\quad \k \preceq \k^* \text{ \textup{(}as ordered fields\textup{)}.}$$
\end{cor}
\begin{proof} The direction $\Rightarrow$ being trivial, we assume
$\k\preceq \k^*$ and shall derive $\Ga\preceq \Ga^*$. By induction on formulas $\phi(x,y)$ (with $x$ and $y$ as above) we show that
for all $\Ga$ and $\Ga^*$ as in the hypothesis of the lemma and
all $c\in \k^m$ and $\gamma\in \Gamma^n$,
\begin{equation}\label{eq:ast}\tag{$\ast$} \Ga\models \phi(c,\gamma)\quad \Longleftrightarrow\quad \Ga^*\models \phi(c,\gamma).  
\end{equation}
For the inductive step, let $\phi=\exists z \theta$, where $\theta=\theta(x,y,z)$ is a formula and $z$ is a single variable of the scalar or vector sort. The direction $\Rightarrow$ in 
\eqref{eq:ast} holds by the (implicit) inductive asumption. Assume 
$\Ga^*\models \phi(c,\gamma)$ where $c\in \k^m$ and $\gamma\in \Gamma^n$. Take a $\kappa$-saturated elementary extension
$\Ga_1$ of $\Ga$, where
$\kappa=|\Gamma^*|^+$. Let $\k_1$ be the scalar field of 
$\Gamma_1$. Then we have an elementary embedding $e\colon \k^*\to \k_1$
that is the identity on $\k$. Proposition~\ref{propeq} (with $\Ga$, $\Ga^*$, $\Ga_1$ in the roles of $\Ga_0$, $\Ga$, $\Ga^*$) gives an embedding $i\colon \Ga^*\to \Ga_1$ where $i_{\operatorname{s}}=e$ and $i_{\operatorname{v}}$ is the identity on $\Gamma$. By the (tacit) inductive hypothesis on $\theta$
we obtain $\Ga_1\models \phi(c,\gamma)$, and thus 
$\Ga\models \phi(c,\gamma)$. 
\end{proof}

\noindent
With~$x$,~$y$ as above, call a formula $\eta(x,y)$ a {\em scalar formula\/}
if it has the form $\zeta\big(s_1(x,y),\dots, s_N(x,y)\big)$ where
$\zeta(z_1,\dots, z_N)$ is a formula in the language of
ordered rings (as specified in (ii) of the description of
$\mathcal{L}_H$), where $z_1,\dots, z_N$ are distinct scalar variables and $s_1(x,y),\dots, s_N(x,y)$ are scalar-valued terms of $\mathcal{L}_H$. 

\begin{theorem}\label{theq} Every formula $\phi(x,y)$ is $T_H$-equivalent to a boolean combination of scalar formulas $\eta(x,y)$ and atomic formulas $\alpha(x,y)$.
\end{theorem}

\noindent
As a consequence, extending $T_H$ by axioms that the scalar field is real closed gives outright QE, without requiring scalar formulas.

\begin{proof} Suppose $(c,\gamma)\in \k^m\times \Gamma^n$ and
$(c^*, \gamma^*)\in (\k^*)^m\times (\Gamma^*)^n$ satisfy the same scalar formulas $\eta(x,y)$ and atomic formulas $\alpha(x,y)$
in $\Ga$ and $\Ga^*$, respectively. It suffices to derive from this
assumption that $(c,\gamma)$ and $(c^*,\gamma^*)$ satisfy the same
formulas in~$\Ga$ and~$\Ga^*$. We may assume that $\Ga^*$ is 
$\kappa$-saturated where $\kappa=|\Gamma|^+$. Let $\Ga_0$ with scalar field $\k_0$ be
the substructure of $\Ga$ generated by $(c,\gamma)$.  Since~$(c,\gamma)$ and $(c^*,\gamma^*)$ realize the same atomic formulas 
$\alpha(x,y)$, we have an embedding $i_0\colon \Ga_0\to \Ga^*$ such that~$i_0(c)=c^*$ and $i_0(\gamma)=\gamma^*$. They also realize the same
scalar formulas~$\eta(x,y)$, so we have an elementary embedding
$e\colon \k\to \k^*$ agreeing with $(i_0)_{\operatorname{s}}$ on~$\k_0$.
Proposition~\ref{propeq} then yields an embedding $i\colon\Ga\to \Ga^*$
extending $i_0$ with $i_{\operatorname{s}}=e$. Then $i$ is an elementary
embedding by Corollary~\ref{coreq}, so $(c,\gamma)$ and $(c^*,\gamma^*)$ do indeed satisfy the same formulas in $\Ga$ and 
$\Ga^*$. 
\end{proof}

\subsection*{Discrete definable sets} We are finally ready to prove the theorem announced in the introduction. We state it here in its natural general setting:

\begin{theorem} Let $\Ga=\big((\Gamma,\psi,P), \k; \sc\!\big)$ be a normalized closed 
$H$-triple of Hahn type 
and let $X\subseteq \Gamma$ be definable in
$\Ga$. Then the following are equivalent:\begin{enumerate}
\item[(i)] $X$ is contained in a finite-dimensional $\k$-linear subspace of $\Gamma$;
\item[(ii)] $X$ is discrete;
\item[(iii)] $X$ has empty interior in $\Gamma$.
\end{enumerate}
\end{theorem}
\begin{proof} The direction (i) $\Rightarrow$ (ii) holds by Lemma~\ref{ovsdis1}. The direction (ii) $\Rightarrow$ (iii) is obvious. (These two implications don't need $X$ to be definable.) 

As to (iii) $\Rightarrow$ (i), assume 
$X$ has empty interior. Take a formula $\phi(y)$ over $\Ga$ in a single vector variable $y$ that defines the set $X$ in $\Ga$.
We use Theorem~\ref{theq} to arrange that $\phi(y)$ is a boolean combination of scalar formulas over $\Ga$ and atomic formulas over
$\Ga$.   
Take a $|\Gamma|^+$-saturated elementary extension 
$\Ga^*=\big((\Gamma^*, \psi^*,P^*),\k^*;\sc^*\!\big)$ of~$\Ga$, and let $X^*\subseteq \Gamma^*$ be defined by $\phi(y)$ in $\Ga^*$.
We identify $\Gamma_{\k^*}$ with $\k^*\Gamma\subseteq \Gamma^*$ in the usual way. 
We Claim~that 
$X^*\subseteq \Gamma_{\k^*}$. (This gives (i) by Lemma~\ref{ovsdis2}.) Consider the substructure $\Ga_{\k^*}=\big((\Gamma_{\k^*},\psi_{\k^*}, P_{\k^*}),\k^*;\sc^*\!\big)$ of $\Ga^*$; it has asymptotic integration by Lemma~\ref{scpr}. Let $X_{\k^*}\subseteq \Gamma_{\k^*}$ be defined in
$\Ga_{\k^*}$ by $\phi(y)$. Then
$X_{\k^*}=X^*\cap \Gamma_{\k^*}$, so our claim
amounts to $X^*=X_{\k^*}$. Suppose towards a contradiction that 
$\gamma^*\in X^*\setminus X_{\k^*}$. In particular,
$\gamma^*\in \Gamma^*\setminus \Gamma_{\k^*}$. 
Saturation yields an $\varepsilon\in \Gamma^*$ such that
$0 < \varepsilon < c^*\gamma$ for all positive $c^*$ in $\k^*$ and
all positive $\gamma\in \Gamma$, so $0 < \varepsilon < \Gamma_{\k^*}^{>}$, and thus
$\Gamma_{\k^*}^{>}$ is not coinitial in $(\Gamma^*)^{>}$.
Lemma~\ref{keyinterval} then yields a $\delta>0$ in $\Gamma^*$ such that all
$\gamma\in \Gamma^*$ with $|\gamma-\gamma^*|< \delta$ yield an isomorphism 
$$\big(\Gamma_{\k^*}\<\gamma^*\>, \psi_{\gamma^*}, P_{\gamma^*}\big)\ \cong\ \big(\Gamma_{\k^*}\<\gamma\>, \psi_{\gamma}, P_{\gamma}\big)\ \subseteq\ \big(\Gamma^*, \psi^*, P^*\big)
$$
of $H$-triples over $\k^*$ sending $\gamma^*$ to $\gamma$.
Hence $s(\gamma^*)=s(\gamma)$ for such $\gamma$ and any scalar-valued $\mathcal{L}_H$-term $s(y)$ over $\Ga$, and
so $\Ga^*\models \phi(\gamma)$ for those $\gamma$. Thus
the interval~$(\gamma^*-\delta, \gamma^*+\delta)$ in $\Gamma^*$
lies entirely in $X^*$, contradicting that $X^*$ is discrete in
$\Gamma^*$.   
\end{proof}

\section{Further Results about Closed $H$-couples}

\noindent
We briefly return to the {\em one-sorted\/} setting of $H$-couples (or $H$-triples) and give two easy applications of Theorem~\ref{hclQE}.   

\subsection*{Definable closure} Let $\Ga^*=(\Gamma^*, \psi^*,P^*)$ be a closed $H$-triple over $\k$. Then we have the notion of the {\em definable closure\/} of a set
$\Gamma\subseteq \Gamma^*$ in $\Ga^*$, and thus of such a set $\Gamma$ being {\em definably closed\/} in $\Ga^*$.
If $\Gamma\subseteq \Gamma^*$ is definably closed in $\Ga^*$, then $\Gamma$ is (the underlying set of)
a subgroup of $\Gamma^*$ with $\psi^*(\Gamma^{\ne})\subseteq \Gamma$, and thus we have 
an $H$-triple $(\Gamma, \psi, P)$ over $\k$ 
with $(\Gamma, \psi, P)\subseteq \Ga^*$.

\begin{prop}\label{dcl} Let $(\Gamma, \psi, P)$ be an $H$-triple over $\k$
with $(\Gamma, \psi, P)\subseteq \Ga$. Then:
$$\text{$\Gamma$ is definably closed in $(\Gamma^*, \psi^*,P^*)$}\quad \Longleftrightarrow\quad    
 \text{$(\Gamma, \psi)$ has asymptotic integration.}$$
\end{prop}
\begin{proof} For $\Rightarrow$, note that
for every $\gamma\in \Gamma$ there is a unique $\alpha\in (\Gamma^*)^{\ne}$ with
$\gamma=\alpha'$.  

For the converse, assume that $(\Gamma, \psi)$ has asymptotic integration  (so $P=\Psi^{\downarrow}$).
Iterating the construction of Lemma~\ref{extpsi} we obtain an increasing continuous chain $$\big((\Gamma_\lambda,\psi_\lambda, P_\lambda)\big)_{\lambda<\nu} \qquad(\text{with $\nu$ an ordinal)}$$ of $H$-triples contained in $(\Gamma^*, \psi^*, P^*)$ as substructures, with $(\Gamma_0, \psi_0, P_0)=(\Gamma, \psi, P)$, such that every $(\Gamma_\lambda,\psi_\lambda, P_\lambda)$ has asymptotic integration with
$P_{\lambda}$ being the downward closure of $\Psi_0$ in $\Gamma_{\lambda}$, and such that the union
$$(\Gamma^{\operatorname{c}}, \psi^{\operatorname{c}}, P^{\operatorname{c}})\ :=\ \bigcup_{\lambda< \nu}(\Gamma_\lambda,\psi_\lambda, P_\lambda)$$
is closed. 
The reference to Lemma~\ref{extpsi} means that for 
$\lambda< \lambda+1<\nu$ we have~$\Gamma_{\lambda+1}=\Gamma_{\lambda}\oplus \k\alpha_{\lambda}$ with $\alpha_{\lambda}>0$ and $\alpha_{\lambda}^\dagger\in P_{\lambda}\setminus \psi_{\lambda}(\Gamma_{\lambda}^{\ne})$. That the chain is continuous means that $(\Gamma_\mu,\psi_\lambda, P_\mu)=\bigcup_{\lambda<\mu}(\Gamma_\lambda,\psi_\lambda, P_\lambda)$
for limit ordinals $\mu< \nu$. Any such~$(\Gamma^{\operatorname{c}}, \psi^{\operatorname{c}}, P^{\operatorname{c}})$ is clearly an 
$H$-closure of $(\Gamma, \psi, P)$, which explains the
superscript~$\operatorname{c}$.  Since $(\Gamma^{\operatorname{c}}, \psi^{\operatorname{c}}, P^{\operatorname{c}})\preceq (\Gamma^*,\psi^*,P^*)$, any element of $\Gamma^*$ that is definable in~$\Ga^*$ over $\Gamma$ must lie in $\Gamma^{\operatorname{c}}$. So let $\gamma^{\operatorname{c}}\in \Gamma^{\operatorname{c}}\setminus \Gamma$; to show that then $\gamma^{\operatorname{c}}$ is not definable in~$\Ga^*$ over $\Gamma$ it suffices by Theorem~\ref{hclQE}
that $\gamma^{\operatorname{c}}$ realizes in $\Ga^*$ the same quantifier-free type over $\Gamma$ as some $\gamma\in \Gamma^{\operatorname{c}}$ with $\gamma\ne \gamma^{\operatorname{c}}$. 
Take $\lambda$ with $\lambda< \lambda+1<\nu$ such that~$\gamma^{\operatorname{c}}\in \Gamma_{\lambda+1}\setminus \Gamma_{\lambda}$. Then
$$\gamma^{\operatorname{c}}\ =\ \gamma_{\lambda}+d\alpha_{\lambda} \qquad (\gamma_{\lambda}\in \Gamma_{\lambda},\ d\in \k^\times).$$
Take any $\alpha\ne \alpha_{\lambda}$ in $\Gamma_{\lambda+1}^{>}$ such that $[\alpha]_{\k}=[\alpha_{\lambda}]_{\k}$. Then $\gamma^{\operatorname{c}}\ne 
\gamma:=\gamma_{\lambda}+d\alpha$.  Lemma~\ref{extpsi}  gives an automorphism $\sigma$ of $(\Gamma_{\lambda+1}, \psi_{\lambda+1}, P_{\lambda+1})$ over $\Gamma_{\lambda}$ with $\sigma(\alpha)=\alpha_{\lambda}$, so
$\sigma(\gamma^{\operatorname{c}})=\gamma$. Thus $\gamma^{\operatorname{c}}$ and $\gamma$ realize in $\Ga^*$ the same quantifier-free type over $\Gamma$.
\end{proof}



\subsection*{A closure property of closed $H$-couples} We show here how \cite[Properties~A and~B]{AvdD} and its variant
 \cite[Section~9.9]{ADH} follow from our QE.

\medskip\noindent
Let $(\Gamma, \psi)$ be an $H$-couple over $\k$. We extend
${\psi\colon \Gamma^{\ne}\to \Gamma}$ to a function $\psi\colon \Gamma_{\infty} \to \Gamma_{\infty}$ by $\psi(0)=\psi(\infty):= \infty$. For $\alpha_1,\dots,\alpha_n\in \Gamma$, $n\ge 1$, we define   
$\psi_{\alpha_1,\dots,\alpha_n}\colon \Gamma_{\infty}\to \Gamma_{\infty}$ by
recursion on $n$: 
$$\psi_{\alpha_1}(\gamma)\ :=\ \psi(\gamma-\alpha_1), \qquad
\psi_{\alpha_1,\dots, \alpha_{n}}(\gamma)\ :=\ \psi\big(\psi_{\alpha_1,\dots, \alpha_{n-1}}(\gamma)-\alpha_{n}\big) 
\text{ for $n\ge 2$.}$$ 
Let $D$ be a subset of an ordered abelian group $\Delta$. Call $D$  {\em bounded\/} if $D\subseteq [p,q]$ for
some $p\le q$ in $\Delta$, and otherwise, call $D$ {\em unbounded}. (These notions and the next one are with respect to the ambient $\Delta$.)
A {\bf (convex) component of $D$} is by 
definition a nonempty convex subset $S$ of $\Delta$ such that 
$S\subseteq D$ and $S$ is maximal with these properties.  \index{component!convex}
The components of $D$ partition the set $D$: for $d\in D$ the unique component of 
$D$ containing $d$ is 
$$\big\{\gamma\in D^{\le d}:\,[\gamma,d]\subseteq D\big\} \cup \big\{\gamma\in D^{\ge d}:\,
[d,\gamma]\subseteq D\big\}.$$ 
Let
$n\ge 1$, and let $\alpha$ be a sequence $\alpha_1,\dots, \alpha_n$ from $\Gamma$. 
We set $$D_{\alpha}\ :=\ \big\{\gamma\in \Gamma:\,\psi_{\alpha}(\gamma)\ne \infty\big\}.$$
Thus 
\begin{align*}
D_{\alpha}\	&=\ \Gamma\setminus \{\alpha_1\} &&\text{for $n=1$, and} \\
D_{\alpha}\	&=\ \big\{\gamma\in D_{\alpha'}:\,\psi_{\alpha'}(\gamma)\ne \alpha_n\big\} &&\text{for $n>1$ and $\alpha'=\alpha_1,\dots, \alpha_{n-1}$.} 
\end{align*}
One checks easily by induction on $n$ that for distinct 
$\gamma, \gamma'\in  D_{\alpha}$,
$$\psi_{\alpha}(\gamma)- \psi_{\alpha}(\gamma')\ =\ o(\gamma-\gamma').$$ 
Let  $n\ge 1$, let $\alpha_1,\dots, \alpha_n\in \Gamma$, set
$\alpha:=(\alpha_1,\dots,\alpha_n)$, and let 
$c_1,\dots, c_n\in \k$.

\medskip\noindent
The next lemma is \cite[Lemma~9.9.3]{ADH}, generalized from 
$\k=\Q$ to arbitrary $\k$, with the same (easy) proof. 

\begin{lemma}\label{strinc} The function 
$$\gamma\mapsto \gamma + c_1\psi_{\alpha_1}(\gamma) + \cdots + c_n\psi_{\alpha_1,\dots,\alpha_n}(\gamma)\ :\,D_{\alpha} \to \Gamma$$
is strictly increasing. Moreover, this function has the intermediate value property on every component of $D_{\alpha}$.
\end{lemma}



\begin{prop}\label{cac} Suppose $(\Gamma, \psi)$ is closed,
$(\Gamma^*,\psi^*)$ is an $H$-couple over $\k$
extending $(\Gamma,\psi)$, and $\gamma\in \Gamma^*$ is such that
\begin{align*} \psi^*_{\alpha_1,\dots,\alpha_n}(\gamma)&\ne \infty\quad \text{\textup{(}so
$\psi^*_{\alpha_1,\dots,\alpha_i}(\gamma)\ne \infty$ for $i=1,\dots,n$\textup{)},}\ \text{ and}\\
\gamma + c_1\psi^*_{\alpha_1}(\gamma) &+ \cdots + c_n\psi^*_{\alpha_1,\dots,\alpha_n}(\gamma)\in \Gamma. \end{align*}
Then $\gamma\in \Gamma$.
\end{prop} 
\begin{proof} By extending $(\Gamma^*, \psi^*)$ we arrange it
to be closed. Then by Theorem~\ref{hclQE},  
$(\Gamma,\psi,\Psi)\preceq (\Gamma^*, \psi^*, \Psi^*)$, and so
we have $\beta\in \Gamma$ such that 
$\psi_{\alpha_1,\dots,\alpha_n}(\beta)\ne \infty$ and
$$\beta + c_1\psi_{\alpha_1}(\beta) + \cdots + c_n\psi_{\alpha_1,\dots,\alpha_n}(\beta)=\gamma + c_1\psi^*_{\alpha_1}(\gamma) + \cdots + c_n\psi^*_{\alpha_1,\dots,\alpha_n}(\gamma).$$ It remains to note that then $\beta=\gamma$ by Lemma~\ref{strinc}.
\end{proof}

\section{Final Remarks} 

\noindent
In \cite{AvdD} we adopted the $2$-sorted setting and ``Hahn type'' 
at the outset, and only observed in its last section that much went through in a one-sorted setting without Hahn type assumption and just rational scalars. Here we have reversed this order, since our proof of Theorem~\ref{thm} required various facts, such as Lemmas~\ref{cutinheritance} and \ref{keyinterval}, about ``one-sorted'' closed $H$-couples over an arbitrary ordered scalar field that are not readily available in \cite{AvdD}. 

There remain several parts in \cite{AvdD} that we have not
tried to cover or extend here. These concern
the definable closure of an $H$-couple in an ambient closed $H$-couple, the uniqueness of 
$H$-closures, the well-orderedness of $\Psi$ for finitely generated $H$-couples, the weak o-minimality of  closed 
$H$-couples, and the local o-minimality and o-minimality at infinity of models of $T_H$.
We alert the reader that our terminology (and notations) concerning asymptotic couples 
 have evolved since \cite{AvdD}, and are now in line with \cite{ADH}, 
and so comparisons with the material here and in \cite{AvdD} require careful attention to the exact meaning of words. 

We do intend to treat some of these topics in a follow-up, since our revisit  also uncovered errors in the 
alleged proofs of weak o-minimality and local o-minimality in~\cite{AvdD}. These can be corrected using the
present paper, but this is not entirely a routine matter.

\end{document}